\documentclass[11pt,letterpaper]{article}
\usepackage{float}
\usepackage{hyperref}
\usepackage{amsmath}
\usepackage{amsfonts}
\usepackage{amsthm}
\usepackage{amssymb}
\usepackage{graphicx}
\usepackage{tikz}
\usepackage[bottom=1in, left=1in, right=1in, top=.75in]{geometry}
\newcommand{\R}{\mathbb{R}}

\newcommand{\Z}{\mathbb{Z}}

\newcommand{\C}{\mathbb{C}}

\title{The squish map and the $\text{SL}_2$ double dimer model}

\author{Leigh Foster\thanks{\href{mailto:leighf@uoregon.edu}{leighf@uoregon.edu}. Leigh Foster was supported by NSF grant DMS-2039316.} \and Benjamin Young\thanks{\href{mailto:bjy@uoregon.edu}{bjy@uoregon.edu}}}

\definecolor{yellow}{RGB}{253, 231, 37}
\definecolor{lime}{RGB}{94, 201, 98}
\definecolor{green}{RGB}{33, 145, 140}
\definecolor{blue}{RGB}{59, 82, 139}
\definecolor{purple}{RGB}{68, 10, 84}

\newtheorem{theorem}{Theorem}
\newtheorem{conjecture}[theorem]{Conjecture}
\newtheorem{corollary}[theorem]{Corollary}

\newtheorem{remark}[theorem]{Remark}

\theoremstyle{definition}

\newtheorem{example}[theorem]{Example}
\newtheorem{definition}[theorem]{Definition}

\DeclareMathOperator{\SL}{SL}

\begin{document}

\maketitle

\begin{abstract}
A plane partition, whose 3D Young diagram is made of unit cubes, can be approximated by a ``coarser" plane partition, made of cubes of side length 2.  Indeed, there are two such approximations obtained by ``rounding up" or ``rounding down" to the nearest cube.  We relate this coarsening (or downsampling) operation to the squish map introduced by the second author in earlier work.  We exhibit a related measure-preserving map between the dimer model on the honeycomb graph, and the $\SL_2$ double dimer model on a coarser honeycomb graph; we compute the most interesting special case of this map, related to plane partition $q$-enumeration with 2-periodic weights.  As an application, we specialize the weights to be certain roots of unity, obtain novel generating functions (some known, some new, and some conjectural) that $(-1)$-enumerate certain classes of pairs of plane partitions according to how their dimer configurations interact. 
\end{abstract}

\section{Introduction}

Anyone who has played with cubical building blocks in their youth has, at some point, constructed a $2 \times 2 \times 2$ cube out of eight $1 \times 1 \times 1$ cubes.  Later in life, some of us (including the authors) went on to study \emph{plane partitions}, which are nothing more than stable piles of cubes in the corner of a large room.  These two experiences tell us that there ought to be \emph{some} relation between plane partitions made out of the little cubes, and plane partitions made out of the bigger ones.  Or, going the other way: what if we take a picture of a plane partition and ``downsample" it, by approximating its $1 \times 1 \times 1$ cubes by roughly one eighth as many $2 \times 2 \times 2$ cubes, as best we can (see Figure~\ref{fig:downsample scad})?  How much information do we lose by doing this?

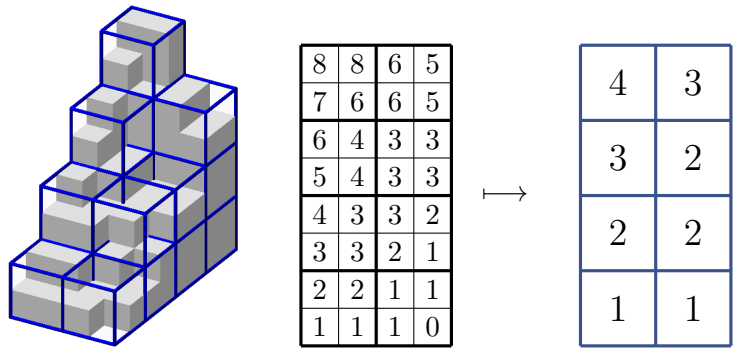
\begin{figure}[ht]
$$
\includegraphics[height=1.8in]{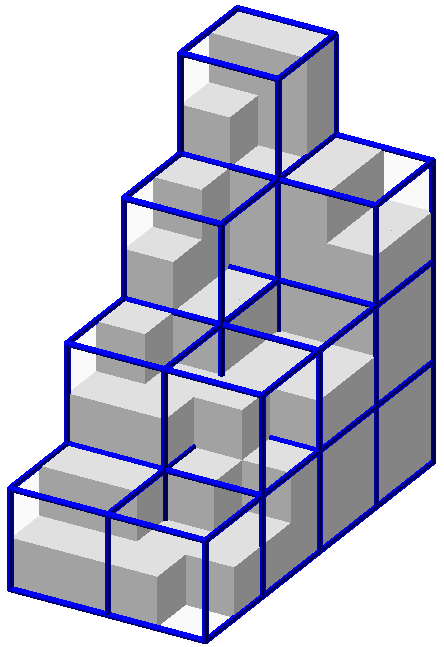}
\qquad
\begin{tikzpicture}[scale = 0.5]
			\draw[very thick] (0, 0) -- (0, 8);
			\draw (1, 0) -- (1, 8);
			\draw[very thick] (2, 0) -- (2, 8);
			\draw (3, 0) -- (3, 8);
			\draw[very thick] (4, 0) -- (4, 8);
			\draw[very thick] (0, 0) -- (4, 0);
			\draw (0, 1) -- (4, 1);
			\draw[very thick] (0, 2) -- (4, 2);
			\draw (0, 3) -- (4, 3);
			\draw[very thick] (0, 4) -- (4, 4);
			\draw (0, 5) -- (4, 5);
			\draw[very thick] (0, 6) -- (4, 6);
			\draw (0, 7) -- (4, 7);
			\draw[very thick] (0, 8) -- (4, 8);
			
			\node at (0.5, 0.5) {1};
			\node at (0.5, 1.5) {2};
			\node at (0.5, 2.5) {3};
			\node at (0.5, 3.5) {4};
			\node at (0.5, 4.5) {5};
			\node at (0.5, 5.5) {6};
			\node at (0.5, 6.5) {7};
			\node at (0.5, 7.5) {8};
			
			\node at (1.5, 0.5) {1};
			\node at (1.5, 1.5) {2};
			\node at (1.5, 2.5) {3};
			\node at (1.5, 3.5) {3};
			\node at (1.5, 4.5) {4};
			\node at (1.5, 5.5) {4};
			\node at (1.5, 6.5) {6};
			\node at (1.5, 7.5) {8};
			
			\node at (2.5, 0.5) {1};
			\node at (2.5, 1.5) {1};
			\node at (2.5, 2.5) {2};
			\node at (2.5, 3.5) {3};
			\node at (2.5, 4.5) {3};
			\node at (2.5, 5.5) {3};
			\node at (2.5, 6.5) {6};
			\node at (2.5, 7.5) {6};
			
			\node at (3.5, 0.5) {0};
			\node at (3.5, 1.5) {1};
			\node at (3.5, 2.5) {1};
			\node at (3.5, 3.5) {2};
			\node at (3.5, 4.5) {3};
			\node at (3.5, 5.5) {3};
			\node at (3.5, 6.5) {5};
			\node at (3.5, 7.5) {5};
			
			\end{tikzpicture}\ \ 
\begin{tikzpicture}[scale = 0.5]
		\node at (-2, 4) {$\longmapsto$};
		\draw[very thick,blue] (0, 0) -- (0, 8);
		\draw[very thick,blue] (2, 0) -- (2, 8);
		\draw[very thick,blue] (4, 0) -- (4, 8);
		\draw[very thick,blue] (0, 0) -- (4, 0);
		\draw[very thick,blue] (0, 2) -- (4, 2);
		\draw[very thick,blue] (0, 4) -- (4, 4);
		\draw[very thick,blue] (0, 6) -- (4, 6);
		\draw[very thick,blue] (0, 8) -- (4, 8);
		
		\node at (1, 1) {\Large 1};
		\node at (1, 3) {\Large 2};
		\node at (1, 5) {\Large 3};
		\node at (1, 7) {\Large 4};
		
		\node at (3, 1) {\Large 1};
		\node at (3, 3) {\Large 2};
		\node at (3, 5) {\Large 2};
		\node at (3, 7) {\Large 3};
		\end{tikzpicture}$$
\caption{A plane partition, viewed as a stack of boxes - and then ``downsampled'' by a factor of two, rounding up to the nearest full box.  The plane partitions are $\pi$ and $\pi_{\text{max}}$, respectively, from Example~\ref{example:downsample} in Section~\ref{subsection:squish map}. }
\label{fig:downsample scad}
\end{figure}

We propose to answer this question through an analysis of the \emph{squish map} -- a map originally studied in \cite{young:squish}, as a means of proving a combinatorial theorem about plane partition enumeration.  Here, we prove that the squish map is a measure-preserving transformation between instances of the single and double dimer models on honeycomb graphs. Thus, the loops of the double dimer model indicate where information is lost in the ``downsampling" process.

\subsection{Definitions and literature survey}
\label{sec:definitions}

A \emph{plane partition} is an infinite matrix of nonnegative integers, all of which are zero sufficiently far from the origin, which are weakly decreasing both in rows or in columns (we do not typically draw the zeros).  Equivalently, interpreting the numbers in a plane partition as $z$ coordinates, one can represent a plane partition as a stack of unit cubes in the corner of a room - this is precisely the relationship between ordinary integer partitions and their Young diagrams, but one dimension higher. 

If the plane partition's cubes fit inside an $x \times y \times z$ box, then it is said to be a \emph{boxed $x \times y \times z$ plane partition}.  MacMahon~\cite{macmahon} proved that the generating function for boxed $x \times y \times z$ plane partitions is
\[
\prod_{i=1}^{x}\prod_{j=1}^{y}\frac{1-q^{i+j+z-1}}{1-q^{i+j-1}}
\]
which, in the limit $x, y, z \rightarrow \infty$ gives the famous generating function for plane partitions,
\[
\prod_{i \geq 1} \left(\frac{1}{1-q^i}\right)^{i}.
\]
Boxed plane partitions are in bijection with perfect matchings on an $x \times y \times z$
hexagon graph, and hence with the dimer model on an certain graph. This result (as far as we know) is folklore, and we don't know a good reference; see Figure~\ref{duploex} for an illustration.  We call the graph in question  
 $H_{x,y,z}$, the $x \times y \times z$ honeycomb graph (see Figure~\ref{single}). Let $D_{x, y, z}$ denote the set of perfect matchings (otherwise known as \emph{dimer configurations}) on $H_{x, y, z}$.  Furthermore, let $DD_{x, y, z}$ be the set of \emph{double dimer configurations} on $H_{x, y, z}$.  That is, an element  $m \in D_{x, y, z}$ is an induced 1-regular subgraph of $H_{x, y, z}$ -- every vertex of $H_{x, y, z}$ is in exactly one edge of $D_{x, y, z}$, whereas an element $m \in DD_{x, y, z}$ is an induced 2-regular subgraph (with the slightly unusual convention that doubled edges are allowed).  See Figure \ref{double}.

We require Kenyon's $\SL_2(\mathbb{C})$-weighted double dimer model~\cite{kenyon:conformal}: in addition to an edge weight, one assigns a $2 \times 2$ matrix $M_e \in \SL_2(\mathbb{C})$ to each edge $e$; the partition function involves a product of traces of such matrices taken around a loop. This collection of matrices is called a \textit{connection} on the graph, by analogy with differential geometry, since its contribution to the partition function is the monodromy around closed paths in the graph.  When all $M_e$ are the identity matrix, all loops contribute $\text{Tr} (I) = 2$, so the model reduces to two independent copies of the ordinary dimer model, but this is typically a much more general and subtle model.  

There are standard, determinantal tools for computing single and double dimer partition functions. For the single dimer model, one uses \cite{kasteleyn:statistics} and \cite{temperley-fisher}; for lattice paths, the references are \cite{lindstrom} and \cite{gessel-viennot}.  Kenyon \cite{kenyon:conformal} introduces a matrix, analogous to Kasteleyn's matrix, whose determinant computes the partition function of this model.  We do not need either determinant for this paper.

The map we introduce here, \textit{the squish map}, is a map from $D_{2x, 2y, 2z}$ to $DD_{x, y, z}$. One can (in principle) discover everything there is to know about it by drawing $H_{2x, 2y, 2z}$ with distorted edge lengths (see Figure \ref{squish}). It was introduced in \cite{young:squish}. Our result here is that the squish map is measure-preserving, for particular choices of the parameters of the single and double dimer models.  Indeed, we are even able to prove a subtle result about the enumeration of down-sampled plane partitions, hinted at in \cite{young:squish}: If we let $x, y, z \rightarrow \infty$, then there is a closed-form generating function which is preserved by the squish map. It is a certain 4-variable generating function for ``$\mathbb{Z}_2 \times \mathbb{Z}_2$ - colored plane partitions" studied in \cite{young:thesis}; in it, a cube at position $(i,j,k)$ gets a color according to the parities of $(i-j)$ and $(i-k)$; the four variables keep track of how many cubes of each color there are.  The squish map allows us to see this generating function as marking certain statistics on a downsampled plane partition. This generating function was studied in \cite{young:thesis} under a certain specialization; we both study it in full generality here, including the original specialization and a few intriguing others.

    \section{The single and double dimer models}

    \begin{definition}
            The \textit{weight} of a graph is an assignment $\nu: E \to \R_{\geq 0}$ of real numbers onto each edge of the graph, where $E$ is the set of edges in a graph $G$. 
    \end{definition} 

    \begin{definition}
            Consider a perfect matching $m$ on $G$. The \textit{weight of m} is defined to be $w(m) = \prod_{e \in m} w(e)$.
            
    \end{definition}
    
    \begin{example}\label{monochromatic}
    	The following weighting reproduces the $Q^{(\text{number of boxes})}$ statistic on plane partitions, up to an overall power of $Q$. This, as far as the authors know, is a ``folklore'' idea for which we do not know a good reference.
    	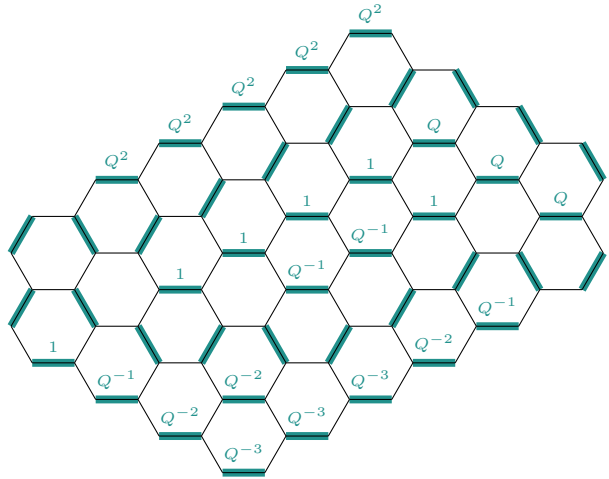
\begin{figure}[ht]
    		\tiny
    		$$\begin{tikzpicture}[scale = 0.8, x=1pt,y=1pt]
    			\draw[line width = 3 pt, green] (100.0, 411.2435565298214) -- node[left]  {}(90.0, 393.92304845413264);
    			\draw[line width = 3 pt, green] (100.0, 376.6025403784438) -- node[left]  {}(90.0, 359.28203230275506);
    			\draw[line width = 3 pt, green] (180.0, 307.3205080756888) -- node[above] {$Q^{-2}$}(160.0, 307.3205080756888);
    			\draw[line width = 3 pt, green] (190.0, 290.0)            -- node[above]  {$Q^{-3}$}(210.0, 290.0);
    			\draw[line width = 3 pt, green] (130.0, 324.6410161513776) -- node[above] {$Q^{-1}$}(150.0, 324.6410161513776);
    			\draw[line width = 3 pt, green] (240.0, 307.3205080756888) -- node[above] {$Q^{-3}$}(220.0, 307.3205080756888);
    			\draw[line width = 3 pt, green] (190.0, 324.6410161513776) -- node[above] {$Q^{-2}$}(210.0, 324.6410161513776);
    			\draw[line width = 3 pt, green] (250.0, 324.6410161513776) -- node[above] {$Q^{-3}$}(270.0, 324.6410161513776);
    			\draw[line width = 3 pt, green] (280.0, 341.9615242270663) -- node[above] {$Q^{-2}$}(300.0, 341.9615242270663);
    			\draw[line width = 3 pt, green] (330.0, 359.28203230275506) -- node[above] {$Q^{-1}$}(310.0, 359.28203230275506);
    			\draw[line width = 3 pt, green] (270.0, 497.8460969082653) -- node[above] {$Q^2$}(250.0, 497.8460969082653);
    			\draw[line width = 3 pt, green] (370.0, 428.5640646055101) -- node[left]  {}(360.0, 445.88457268119896);
    			\draw[line width = 3 pt, green] (220.0, 480.52558883257643) -- node[above] {$Q^2$}(240.0, 480.52558883257643);
    			\draw[line width = 3 pt, green] (330.0, 428.5640646055101) -- node[above] {$Q$}(310.0, 428.5640646055101);
    			\draw[line width = 3 pt, green] (330.0, 463.2050807568877) -- node[left]  {}(340.0, 445.88457268119896);
    			\draw[line width = 3 pt, green] (240.0, 411.2435565298214) -- node[above] {$1$}(220.0, 411.2435565298214);
    			\draw[line width = 3 pt, green] (300.0, 411.2435565298214) -- node[above] {$1$}(280.0, 411.2435565298214);
    			\draw[line width = 3 pt, green] (300.0, 376.6025403784438) -- node[left]  {}(310.0, 393.92304845413264);
    			\draw[line width = 3 pt, green] (210.0, 428.5640646055101) -- node[left]  {}(220.0, 445.88457268119896);
    			\draw[line width = 3 pt, green] (210.0, 463.2050807568877) -- node[above] {$Q^2$}(190.0, 463.2050807568877);
    			\draw[line width = 3 pt, green] (130.0, 359.28203230275506) -- node[left] {}(120.0, 376.6025403784438);
    			\draw[line width = 3 pt, green] (190.0, 393.92304845413264) -- node[above] {1}(210.0, 393.92304845413264);
    			\draw[line width = 3 pt, green] (180.0, 445.88457268119896) -- node[above] {$Q^2$}(160.0, 445.88457268119896);
    			\draw[line width = 3 pt, green] (180.0, 411.2435565298214) -- node[left]  {}(190.0, 428.5640646055101);
    			\draw[line width = 3 pt, green] (270.0, 359.28203230275506) -- node[left] {}(280.0, 376.6025403784438);
    			\draw[line width = 3 pt, green] (250.0, 393.92304845413264) -- node[above] {$Q^{-1}$}(270.0, 393.92304845413264);
    			\draw[line width = 3 pt, green] (120.0, 411.2435565298214) -- node[left]  {}(130.0, 393.92304845413264);
    			\draw[line width = 3 pt, green] (160.0, 411.2435565298214) -- node[left]  {}(150.0, 393.92304845413264);
    			\draw[line width = 3 pt, green] (150.0, 428.5640646055101) -- node[above] {$Q^2$}(130.0, 428.5640646055101);
    			\draw[line width = 3 pt, green] (250.0, 463.2050807568877) -- node[left]  {}(240.0, 445.88457268119896);
    			\draw[line width = 3 pt, green] (270.0, 428.5640646055101) -- node[above] {1}(250.0, 428.5640646055101);
    			\draw[line width = 3 pt, green] (300.0, 445.88457268119896) -- node[above] {$Q$}(280.0, 445.88457268119896);
    			\draw[line width = 3 pt, green] (280.0, 480.52558883257643) -- node[left] {}(270.0, 463.2050807568877);
    			\draw[line width = 3 pt, green] (310.0, 463.2050807568877) -- node[left]  {}(300.0, 480.52558883257643);
    			\draw[line width = 3 pt, green] (220.0, 341.9615242270663) -- node[left]  {}(210.0, 359.28203230275506);
    			\draw[line width = 3 pt, green] (240.0, 376.6025403784438) -- node[above] {$Q^{-1}$}(220.0, 376.6025403784438);
    			\draw[line width = 3 pt, green] (240.0, 341.9615242270663) -- node[left]  {}(250.0, 359.28203230275506);
    			\draw[line width = 3 pt, green] (160.0, 341.9615242270663) -- node[left]  {}(150.0, 359.28203230275506);
    			\draw[line width = 3 pt, green] (180.0, 376.6025403784438) -- node[above] {$1$}(160.0, 376.6025403784438);
    			\draw[line width = 3 pt, green] (190.0, 359.28203230275506) -- node[left] {}(180.0, 341.9615242270663);
    			\draw[line width = 3 pt, green] (330.0, 393.92304845413264) -- node[left] {}(340.0, 376.6025403784438);
    			\draw[line width = 3 pt, green] (360.0, 411.2435565298214) -- node[above] {$Q$}(340.0, 411.2435565298214);
    			\draw[line width = 3 pt, green] (360.0, 376.6025403784438) -- node[left]  {}(370.0, 393.92304845413264);
    			\draw[line width = 3 pt, green] (100.0, 341.9615242270663) -- node[above] {$1$}(120.0, 341.9615242270663);
    			\draw (100.0, 341.9615242270663) -- (120.0, 341.9615242270663);
    			\draw (100.0, 341.9615242270663) -- (90.0, 359.28203230275506);
    			\draw (100.0, 376.6025403784438) -- (90.0, 359.28203230275506);
    			\draw (100.0, 376.6025403784438) -- (120.0, 376.6025403784438);
    			\draw (100.0, 376.6025403784438) -- (90.0, 393.92304845413264);
    			\draw (90.0, 393.92304845413264) -- (100.0, 411.2435565298214);
    			\draw (120.0, 411.2435565298214) -- (100.0, 411.2435565298214);
    			\draw (130.0, 324.6410161513776) -- (120.0, 341.9615242270663);
    			\draw (120.0, 341.9615242270663) -- (130.0, 359.28203230275506);
    			\draw (130.0, 359.28203230275506) -- (120.0, 376.6025403784438);
    			\draw (130.0, 393.92304845413264) -- (120.0, 376.6025403784438);
    			\draw (120.0, 411.2435565298214) -- (130.0, 393.92304845413264);
    			\draw (120.0, 411.2435565298214) -- (130.0, 428.5640646055101);
    			\draw (180.0, 307.3205080756888) -- (160.0, 307.3205080756888);
    			\draw (150.0, 324.6410161513776) -- (160.0, 307.3205080756888);
    			\draw (130.0, 324.6410161513776) -- (150.0, 324.6410161513776);
    			\draw (160.0, 341.9615242270663) -- (150.0, 324.6410161513776);
    			\draw (160.0, 341.9615242270663) -- (180.0, 341.9615242270663);
    			\draw (160.0, 341.9615242270663) -- (150.0, 359.28203230275506);
    			\draw (130.0, 359.28203230275506) -- (150.0, 359.28203230275506);
    			\draw (160.0, 376.6025403784438) -- (150.0, 359.28203230275506);
    			\draw (180.0, 376.6025403784438) -- (160.0, 376.6025403784438);
    			\draw (160.0, 376.6025403784438) -- (150.0, 393.92304845413264);
    			\draw (130.0, 393.92304845413264) -- (150.0, 393.92304845413264);
    			\draw (160.0, 411.2435565298214) -- (150.0, 393.92304845413264);
    			\draw (180.0, 411.2435565298214) -- (160.0, 411.2435565298214);
    			\draw (150.0, 428.5640646055101) -- (160.0, 411.2435565298214);
    			\draw (150.0, 428.5640646055101) -- (130.0, 428.5640646055101);
    			\draw (150.0, 428.5640646055101) -- (160.0, 445.88457268119896);
    			\draw (180.0, 445.88457268119896) -- (160.0, 445.88457268119896);
    			\draw (180.0, 307.3205080756888) -- (190.0, 290.0);
    			\draw (180.0, 307.3205080756888) -- (190.0, 324.6410161513776);
    			\draw (190.0, 324.6410161513776) -- (180.0, 341.9615242270663);
    			\draw (190.0, 359.28203230275506) -- (180.0, 341.9615242270663);
    			\draw (180.0, 376.6025403784438) -- (190.0, 359.28203230275506);
    			\draw (180.0, 376.6025403784438) -- (190.0, 393.92304845413264);
    			\draw (180.0, 411.2435565298214) -- (190.0, 393.92304845413264);
    			\draw (180.0, 411.2435565298214) -- (190.0, 428.5640646055101);
    			\draw (180.0, 445.88457268119896) -- (190.0, 428.5640646055101);
    			\draw (180.0, 445.88457268119896) -- (190.0, 463.2050807568877);
    			\draw (190.0, 290.0) -- (210.0, 290.0);
    			\draw (210.0, 290.0) -- (220.0, 307.3205080756888);
    			\draw (240.0, 307.3205080756888) -- (220.0, 307.3205080756888);
    			\draw (210.0, 324.6410161513776) -- (220.0, 307.3205080756888);
    			\draw (190.0, 324.6410161513776) -- (210.0, 324.6410161513776);
    			\draw (220.0, 341.9615242270663) -- (210.0, 324.6410161513776);
    			\draw (220.0, 341.9615242270663) -- (240.0, 341.9615242270663);
    			\draw (220.0, 341.9615242270663) -- (210.0, 359.28203230275506);
    			\draw (190.0, 359.28203230275506) -- (210.0, 359.28203230275506);
    			\draw (220.0, 376.6025403784438) -- (210.0, 359.28203230275506);
    			\draw (240.0, 376.6025403784438) -- (220.0, 376.6025403784438);
    			\draw (220.0, 376.6025403784438) -- (210.0, 393.92304845413264);
    			\draw (190.0, 393.92304845413264) -- (210.0, 393.92304845413264);
    			\draw (210.0, 393.92304845413264) -- (220.0, 411.2435565298214);
    			\draw (240.0, 411.2435565298214) -- (220.0, 411.2435565298214);
    			\draw (210.0, 428.5640646055101) -- (220.0, 411.2435565298214);
    			\draw (210.0, 428.5640646055101) -- (190.0, 428.5640646055101);
    			\draw (210.0, 428.5640646055101) -- (220.0, 445.88457268119896);
    			\draw (220.0, 445.88457268119896) -- (240.0, 445.88457268119896);
    			\draw (210.0, 463.2050807568877) -- (220.0, 445.88457268119896);
    			\draw (210.0, 463.2050807568877) -- (190.0, 463.2050807568877);
    			\draw (220.0, 480.52558883257643) -- (210.0, 463.2050807568877);
    			\draw (220.0, 480.52558883257643) -- (240.0, 480.52558883257643);
    			\draw (240.0, 307.3205080756888) -- (250.0, 324.6410161513776);
    			\draw (250.0, 324.6410161513776) -- (240.0, 341.9615242270663);
    			\draw (240.0, 341.9615242270663) -- (250.0, 359.28203230275506);
    			\draw (240.0, 376.6025403784438) -- (250.0, 359.28203230275506);
    			\draw (240.0, 376.6025403784438) -- (250.0, 393.92304845413264);
    			\draw (240.0, 411.2435565298214) -- (250.0, 393.92304845413264);
    			\draw (240.0, 411.2435565298214) -- (250.0, 428.5640646055101);
    			\draw (250.0, 428.5640646055101) -- (240.0, 445.88457268119896);
    			\draw (250.0, 463.2050807568877) -- (240.0, 445.88457268119896);
    			\draw (250.0, 463.2050807568877) -- (240.0, 480.52558883257643);
    			\draw (240.0, 480.52558883257643) -- (250.0, 497.8460969082653);
    			\draw (250.0, 324.6410161513776) -- (270.0, 324.6410161513776);
    			\draw (280.0, 341.9615242270663) -- (270.0, 324.6410161513776);
    			\draw (280.0, 341.9615242270663) -- (300.0, 341.9615242270663);
    			\draw (270.0, 359.28203230275506) -- (280.0, 341.9615242270663);
    			\draw (270.0, 359.28203230275506) -- (250.0, 359.28203230275506);
    			\draw (270.0, 359.28203230275506) -- (280.0, 376.6025403784438);
    			\draw (300.0, 376.6025403784438) -- (280.0, 376.6025403784438);
    			\draw (280.0, 376.6025403784438) -- (270.0, 393.92304845413264);
    			\draw (250.0, 393.92304845413264) -- (270.0, 393.92304845413264);
    			\draw (270.0, 393.92304845413264) -- (280.0, 411.2435565298214);
    			\draw (300.0, 411.2435565298214) -- (280.0, 411.2435565298214);
    			\draw (270.0, 428.5640646055101) -- (280.0, 411.2435565298214);
    			\draw (270.0, 428.5640646055101) -- (250.0, 428.5640646055101);
    			\draw (270.0, 428.5640646055101) -- (280.0, 445.88457268119896);
    			\draw (300.0, 445.88457268119896) -- (280.0, 445.88457268119896);
    			\draw (270.0, 463.2050807568877) -- (280.0, 445.88457268119896);
    			\draw (250.0, 463.2050807568877) -- (270.0, 463.2050807568877);
    			\draw (280.0, 480.52558883257643) -- (270.0, 463.2050807568877);
    			\draw (280.0, 480.52558883257643) -- (300.0, 480.52558883257643);
    			\draw (280.0, 480.52558883257643) -- (270.0, 497.8460969082653);
    			\draw (270.0, 497.8460969082653) -- (250.0, 497.8460969082653);
    			\draw (310.0, 359.28203230275506) -- (300.0, 341.9615242270663);
    			\draw (300.0, 376.6025403784438) -- (310.0, 359.28203230275506);
    			\draw (300.0, 376.6025403784438) -- (310.0, 393.92304845413264);
    			\draw (300.0, 411.2435565298214) -- (310.0, 393.92304845413264);
    			\draw (300.0, 411.2435565298214) -- (310.0, 428.5640646055101);
    			\draw (300.0, 445.88457268119896) -- (310.0, 428.5640646055101);
    			\draw (310.0, 463.2050807568877) -- (300.0, 445.88457268119896);
    			\draw (310.0, 463.2050807568877) -- (300.0, 480.52558883257643);
    			\draw (330.0, 359.28203230275506) -- (310.0, 359.28203230275506);
    			\draw (330.0, 359.28203230275506) -- (340.0, 376.6025403784438);
    			\draw (360.0, 376.6025403784438) -- (340.0, 376.6025403784438);
    			\draw (330.0, 393.92304845413264) -- (340.0, 376.6025403784438);
    			\draw (330.0, 393.92304845413264) -- (310.0, 393.92304845413264);
    			\draw (330.0, 393.92304845413264) -- (340.0, 411.2435565298214);
    			\draw (360.0, 411.2435565298214) -- (340.0, 411.2435565298214);
    			\draw (330.0, 428.5640646055101) -- (340.0, 411.2435565298214);
    			\draw (330.0, 428.5640646055101) -- (310.0, 428.5640646055101);
    			\draw (330.0, 428.5640646055101) -- (340.0, 445.88457268119896);
    			\draw (360.0, 445.88457268119896) -- (340.0, 445.88457268119896);
    			\draw (330.0, 463.2050807568877) -- (340.0, 445.88457268119896);
    			\draw (310.0, 463.2050807568877) -- (330.0, 463.2050807568877);
    			\draw (360.0, 376.6025403784438) -- (370.0, 393.92304845413264);
    			\draw (360.0, 411.2435565298214) -- (370.0, 393.92304845413264);
    			\draw (370.0, 428.5640646055101) -- (360.0, 411.2435565298214);
    			\draw (370.0, 428.5640646055101) -- (360.0, 445.88457268119896);
    		\end{tikzpicture}$$
    		
    		\caption{A perfect matching (dimer configuration) on $H_{6, 2, 4}$, with a monochromatic weighting on the horizontal edges. Note that although the edges labeled 1 are lined up along a particular diagonal, this does not have to be this case.}
    		\label{single}
    	\end{figure}
Assign a weight of 1 to one horizontal edge in each column of hexagons. Above each horizontal edge with weight 1, assign weights $Q, Q^2, Q^3, \ldots$, and below weights $Q^{-1}, Q^{-2}, \ldots$.  Non-horizontal edges get weight 1.  Then the overall weight of the perfect matching in Figure~\ref{single} is $Q^{-6}$.
\end{example}

    \begin{figure}[ht]
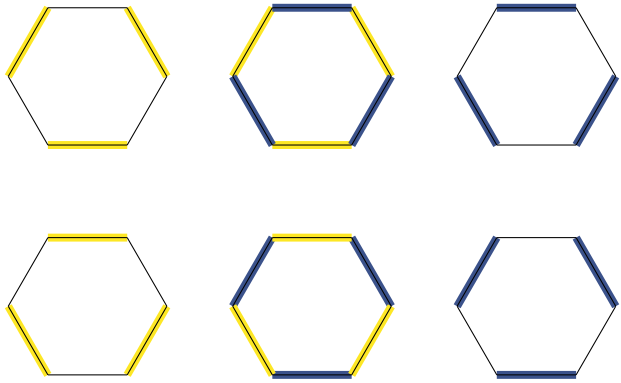

            $$

            $$
            
            \caption{The two possible configurations of a single hexagon loop in the double dimer model}
            \label{twoloops}
    \end{figure}

 \begin{figure}[ht]
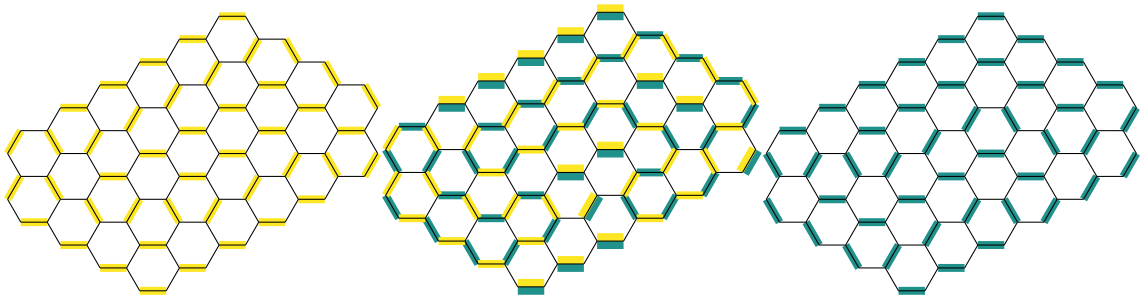

            $$%
$$
            
            \caption{A double dimer configuration -- an overlaying of the right and the left single dimer configurations, that consists of closed loops and doubled edges}
            \label{double}
    \end{figure}

    The following definition of the $\SL_2(\mathbb{C})$ double dimer model is due to Kenyon~\cite{kenyon:conformal}:
    \begin{definition}
            Let $G = (V,E)$ be a bipartite graph with a \emph{scalar weight} $w:E \rightarrow \C$ as well as an \emph{$\SL_2$ connection}: a map $\Gamma:E \rightarrow SL_2(\C)$. Then the contribution of a double dimer configuration $DD$ is defined to be 
            \[
            \left(\prod_{e \in m} w(e) \right)
            \times\!\!\!\!\!\! \prod_{\text{closed loops $L$}}
            \text{Tr}\left(\prod_{e \in L} \Gamma(e)\right).\]
    \end{definition}

    \section{The squish map}

    \subsection{Coordinates on the honeycomb grid}

    The honeycomb graph is the tiling of the plane by hexagons; the center of each hexagon is a point in the dual triangular lattice.  A convenient way to give coordinates to the triangular lattice is to draw it on the plane normal to $(1,1,1) \in \mathbb{R}^3$.   This plane has an orthonormal basis
    $$
    \vec{x} = \frac{1}{\sqrt{2}}(-1,1,0) \qquad
    \vec{y} = \frac{1}{\sqrt{6}}\left(
    -1,-1,2
    \right)$$which we won't really use, but we \emph{will} orient our pictures according to it, and we will use words like up, left, bottom, horizontal, etc... in the conventional way with respect to this basis. 

    The set $\mathbb{Z}^3 \subseteq \mathbb{R}^3$ projects onto a copy of the triangular lattice in this plane. Draw the honeycomb graph $G$ in such a way that the lattice points are the centers of the hexagons, then label each lattice point (and thus each hexagon) with any of the lattice points that project to its center.  For instance, the hexagon at the origin has the labels $(0,0,0), \pm(1,1,1), \pm(2,2,2),$ etc.  Call this particular embedding of the honeycomb graph $H$.   In the $\vec{x}, \vec{y}$ coordinates above, the edge common to the hexagon at $(0,0,0)$ and the hexagon at $(0,0,1)$ is horizontal.

    There is also a second embedding of the honeycomb graph onto the plane which is of interest to us, which we shall glibly call $2H$, and it is obtained by projecting the double-sized lattice $(2\Z)\times(2\Z)\times(2\Z)$ to a (bigger) triangular lattice in the plane, and taking the planar dual.  The hexagons of $2H$ have even coordinates $(2i, 2j, 2k)$.

    \subsection{Degenerating H to 2H}
    The squish map can be defined entirely combinatorially, but for visualization purposes it is extremely helpful to first define a continuous degeneration $H(t):[0,1] \rightarrow \mathbb{R}^2$ such that $H(0) = H$ and $H(1) = 2H$.  This degeneration is shown in Figure~\ref{squish} and first appeared in \cite{young:squish}.  

    To define $H(t)$, write $H=(V,E)$, and let $H^{ev}=(V^{ev}, E^{ev})$ be the subgraph of $H$ consisting of all hexagons whose centers have even coordinates; we say that $H^{ev}$ are the ``even hexagons".  Let $P = H \setminus H^{ev}$.  Then the graph $(V,P)$ is a disconnected union of ``propellers" (or $K_{1,3}$'s or ``claws", depending on what dialect of graph theory you speak).  Each propeller has a central vertex $v$ and three leaves $x, y, z$.   For $t \in [0,1]$, define $x(t)=tv + (1-t)x$, so that $x(t)$ is a point on the edge joining $v$ to $x$.  Define $y(t)$ and $z(t)$ similarly, and make the same definitions at each other propeller.  

    The graph embedding $H(t)$ is obtained by drawing the vertex $a$ at position $a(t)$.  For $0 \leq t < 1$, $H(t)$ is an embedding of $H$; indeed, $H(t)$ is an explicit homotopy equivalence.  At $t=1$, however, we have $v=x=y=z$, and thus $H(1)$ degenerates to an embedding of $2H$.  

    \subsection{The squish map}\label{subsection:squish map}

    $H(t)$ defines a 2-to-1 map $Sq$ from $E^{ev}$ to the edges of $2H$: given an edge $e$ of an even hexagon in $H$, find the corresponding edge in $H(0)$, and let $Sq(e)$ be the corresponding edge in $H(1)$. 
    
    \begin{definition}
    The squish map $Sq: D(H) \rightarrow DD(2H)$ sends a dimer configuration $m$ on $H$ to a double dimer configuration $Sq(m)$ on $2H$ as follows: let $m^{ev} = m \cap E^{ev}$; then $Sq(m) = Sq(m^{ev})$.
    \end{definition}

    Indeed, when drawing $m$ on the graph $H(t)$ for $t < 1$, visualize what the squish map is doing: propeller edges of $m$ get shorter and shorter, while the doubled edges get longer and closer together.  This process is shown in Figure~\ref{squish}.

            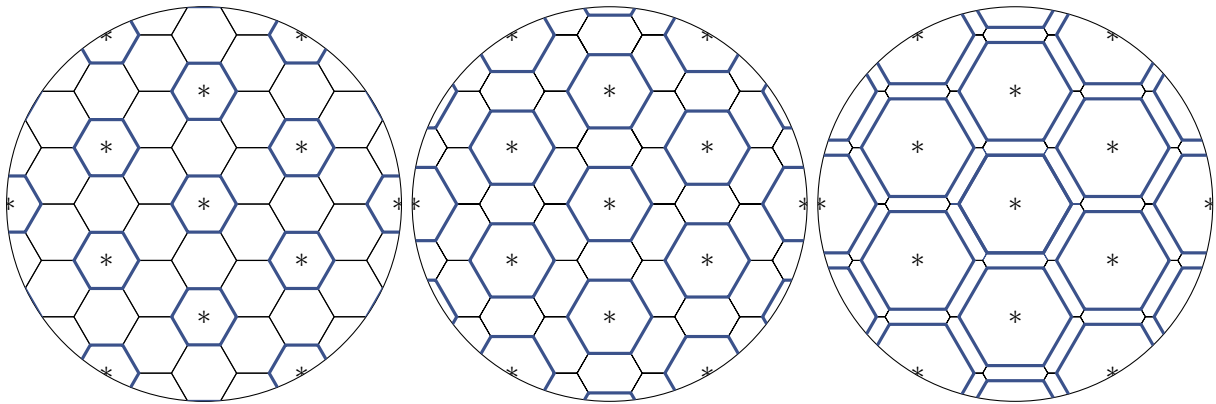
\begin{figure}[ht]		
                     \begin{tikzpicture}[scale = 0.75]
                    	\def\hexagon{+(0:0.575cm) -- +(60:0.575cm) -- +(120:0.575cm) -- +(180:0.575cm) -- +(240:0.575cm) -- +(300:0.575cm) -- cycle}

                    	\clip[draw] (1.73,1) circle (3.5cm);
                    	
                    	\foreach \x in {-5,...,5}
                    	\foreach \y\ in {-5,...,5}
                    	\draw (1.73*\x,\y) \hexagon;

                    	\foreach \x in {-5,...,5}
                    	\foreach \y\ in {-5,...,5}
                    	\draw (1.73*\x + 1.73*0.5,\y + 0.5) \hexagon;

                    	\foreach \x in {-5,...,5}
                    	\foreach \y\ in {-5,...,5}{%
                    		\draw [very thick, blue] (2 * 1.73*\x,2 * \y) \hexagon;
                    		\node at (2 * 1.73*\x,2 * \y) {$*$};
                    	}

                    	\foreach \x in {-5,...,5}
                    	\foreach \y\ in {-5,...,5}{%
                    		\draw[very thick, blue] (2 * 1.73*\x + 1.73,2 * \y + 1) \hexagon;
                    		\node at (2 * 1.73*\x + 1.73,2 * \y + 1) {$*$};
                    	}

                    \end{tikzpicture}
                    \begin{tikzpicture}[scale = 0.75]
                    	\def\hexagon{+(0:0.75cm) -- +(60:0.75cm) -- +(120:0.75cm) -- +(180:0.75cm) -- +(240:0.75cm) -- +(300:0.75cm) -- cycle}
                    	\def\hexagonweights{+(0:1cm) -- node [anchor = east]{$C$} +(60:1cm) -- node [anchor = north]{$A^{-1}$} +(120:1cm) -- node [anchor = west]{$B^{-1}$} +(180:1cm) -- node [anchor = west]{$C^{-1}$} +(240:1cm) -- node [anchor = south]{$A$} +(300:1cm) -- node [anchor = east]{$B$}cycle;}
                    	\def \squish {+(0: 0.75cm) -- 
                    		++(60: 0.75cm) -- +(60: .42cm) -- +(240: 0cm) -- 
                    		++(180: 0.75cm) -- +(120: .42cm) -- +(240: 0cm) -- 
                    		++(240: 0.75cm) -- +(180: .42cm) -- + (240: 0cm) -- 
                    		++(300: 0.75cm) -- +(240: .42cm) -- +(240: 0cm) -- 
                    		++(360: 0.75cm) -- +(300: .42cm) -- +(240: 0cm) -- 
                    		++(420: 0.75cm) -- +(360: .42cm) -- +(240: 0cm)}
                    	
                    	\clip[draw] (1.73,1) circle (3.5cm);
                    	\draw (1.73, 1) \squish;
                    	\draw[] (3.46, 2) \squish;
                    	\draw (1.73, 3) \squish;
                    	\draw (1.73, -1) \squish;
                    	\draw (0,0) \squish;
                    	\draw (1.73, 3) \squish;
                    	\draw (0, 2) \squish;
                    	\draw (-1.73, 1) \squish;
                    	\draw(3.46, 0) \squish;
                    	\draw(0, -2) \squish;
                    	\draw(0,4) \squish;
                    	\draw(-1.73, 3) \squish;
                    	\draw(-1.73, -1) \squish;
                    	\draw(1.73, 5) \squish;
                    	\draw(1.73, -3) \squish;
                    	\draw (3.46, 4) \squish;
                    	\draw (3.46, -2) \squish;
                    	\draw (3.46 + 1.73, 1) \squish;
                    	\draw (3.46 + 1.73, 3) \squish;
                    	\draw (3.46 + 1.73, -1) \squish;
                    	
                    	\foreach \x in {-5,...,5}
                    	\foreach \y\ in {-5,...,5}{%
                    		\draw [very thick, blue] (2 * 1.73*\x,2 * \y) \hexagon;
                    		\node at (2 * 1.73*\x,2 * \y) {$*$};
                    	}

                    	\foreach \x in {-5,...,5}
                    	\foreach \y\ in {-5,...,5}{%
                    		\draw[very thick, blue] (2 * 1.73*\x + 1.73,2 * \y + 1) \hexagon;
                    		\node at (2 * 1.73*\x + 1.73,2 * \y + 1) {$*$};
                    	}
                    \end{tikzpicture}
                    \begin{tikzpicture}[scale = 0.75]
                    	\def\hexagon{+(0:1cm) -- +(60:1cm) -- +(120:1cm) -- +(180:1cm) -- +(240:1cm) -- +(300:1cm) -- cycle}
                    	\def\hexagonweights{+(0:1cm) -- node [anchor = east]{$C$} +(60:1cm) -- node [anchor = north]{$A^{-1}$} +(120:1cm) -- node [anchor = west]{$B^{-1}$} +(180:1cm) -- node [anchor = west]{$C^{-1}$} +(240:1cm) -- node [anchor = south]{$A$} +(300:1cm) -- node [anchor = east]{$B$}cycle;}
                    	\def \squish {+(0: 1cm) -- 
                    		++(60: 1cm) -- +(60: .155cm) -- +(240: 0cm) -- 
                    		++(180:1cm) -- +(120: .155cm) -- +(240: 0cm) -- 
                    		++(240: 1cm) -- +(180: .155cm) -- + (240: 0cm) -- 
                    		++(300: 1cm) -- +(240: .155cm) -- +(240: 0cm) -- 
                    		++(360: 1cm) -- +(300: .155cm) -- +(240: 0cm) -- 
                    		++(420: 1cm) -- +(360: .155cm) -- +(240: 0cm)}
                    	
                    	\clip[draw] (1.73,1) circle (3.5cm);
                    	\draw[very thick, color = blue] (1.73, 1) \hexagon;
                    	\draw[color = blue] (1.73, 1) \squish;
                    	\draw[] (3.46, 2) \squish;
                    	\draw (1.73, 3) \squish;
                    	\draw (1.73, -1) \squish;
                    	\draw (0,0) \squish;
                    	\draw (1.73, 3) \squish;
                    	\draw (0, 2) \squish;
                    	\draw (-1.73, 1) \squish;
                    	\draw(3.46, 0) \squish;
                    	\draw(0, -2) \squish;
                    	\draw(0,4) \squish;
                    	\draw(-1.73, 3) \squish;
                    	\draw(-1.73, -1) \squish;
                    	\draw(1.73, 5) \squish;
                    	\draw(1.73, -3) \squish;
                    	\draw (3.46, 4) \squish;
                    	\draw (3.46, -2) \squish;
                    	\draw (3.46 + 1.73, 1) \squish;
                    	\draw (3.46 + 1.73, 3) \squish;
                    	\draw (3.46 + 1.73, -1) \squish;
                    	
                    	\foreach \x in {-5,...,5}
                    	\foreach \y\ in {-5,...,5}{%
                    		\draw [very thick, blue] (2 * 1.73*\x,2 * \y) \hexagon;
                    		\node at (2 * 1.73*\x,2 * \y) {$*$};
                    	}

                    	\foreach \x in {-5,...,5}
                    	\foreach \y\ in {-5,...,5}{%
                    		\draw[very thick, blue] (2 * 1.73*\x + 1.73,2 * \y + 1) \hexagon;
                    		\node at (2 * 1.73*\x + 1.73,2 * \y + 1) {$*$};
                    	}
                    	
                    \end{tikzpicture}
                    
                    \caption {The degeneration H(t)}
                    \label{squish}
            \end{figure}

	We now relate the squish map to the $1 \times 1 \times 1$ and $2 \times 2 \times 2$ stacked cubes mentioned in the introduction. We can use the language of plane partitions to visualize a perfect matching in the single dimer model as a stack of cubes in the corner of a room, as in Figure \ref{duploex}. Then putting our matching through the squish map `downsamples' the plane partition: we would use 1/8 as many $2 \times 2 \times 2$ bricks instead. For example, to build a $2 \times 2 \times 4$ prism we can use 16 single cubes, or two larger $2 \times 2 \times 2$ cubes. 
	
	\begin{theorem}
		\label{duplo}
		Consider a single dimer configuration $\mathcal{S}$ sent through the squish map resulting in a double dimer configuration $\mathcal{D}$ made of loops and doubled edges. In each $2 \times 2$ section of the plane partition for $\mathcal{S}$, let $\pi_{min}\left(\!\!\raisebox{-0.48cm}{
                               \begin{tikzpicture}[scale = 0.55]
		\draw[very thick] (0, 0) -- (0, 2);
		\draw (1, 0) -- (1, 2);
		\draw[very thick] (2, 0) -- (2, 2);
		\draw[very thick] (0, 0) -- (2, 0);
		\draw (0, 1) -- (2, 1);
		\draw[very thick] (0, 2) -- (2, 2);
		\node at (0.5, 1.5) {m};
		\node at (0.5, 0.5) {o};
		
		\node at (1.5, 1.5) {n};
		\node at (1.5, 0.5) {p};
	\end{tikzpicture}} \right) = \left\lfloor{\frac{\min(m,n,o,p)}{2}}\right\rfloor$, and $\pi_{max}\left(\!\!
       \raisebox{-0.48cm}{
                               \begin{tikzpicture}[scale = 0.55]
		\draw[very thick] (0, 0) -- (0, 2);
		\draw (1, 0) -- (1, 2);
		\draw[very thick] (2, 0) -- (2, 2);
		\draw[very thick] (0, 0) -- (2, 0);
		\draw (0, 1) -- (2, 1);
		\draw[very thick] (0, 2) -- (2, 2);
		\node at (0.5, 1.5) {m};
		\node at (0.5, 0.5) {o};
		
		\node at (1.5, 1.5) {n};
		\node at (1.5, 0.5) {p};
	\end{tikzpicture}} \right) = \left\lceil{\frac{\max(m,n,o,p)}{2}}\right\rceil$.
		
		These are the minimal and maximal plane partitions that, when overlayed, give us the double dimer configuration $\mathcal{D}$.  
	\end{theorem}

	A few natural questions one might ask at this point would be: what single dimer configurations squish to a given double dimer configuration? In the language of building blocks, how can we tell which larger $2 \times 2 \times 2$ brick configurations overlay with one another to give us a particular double dimer model? How do we recover that information if we were only given the $1 \times 1 \times 1$ smaller blocks?
	
	\begin{example}
		
		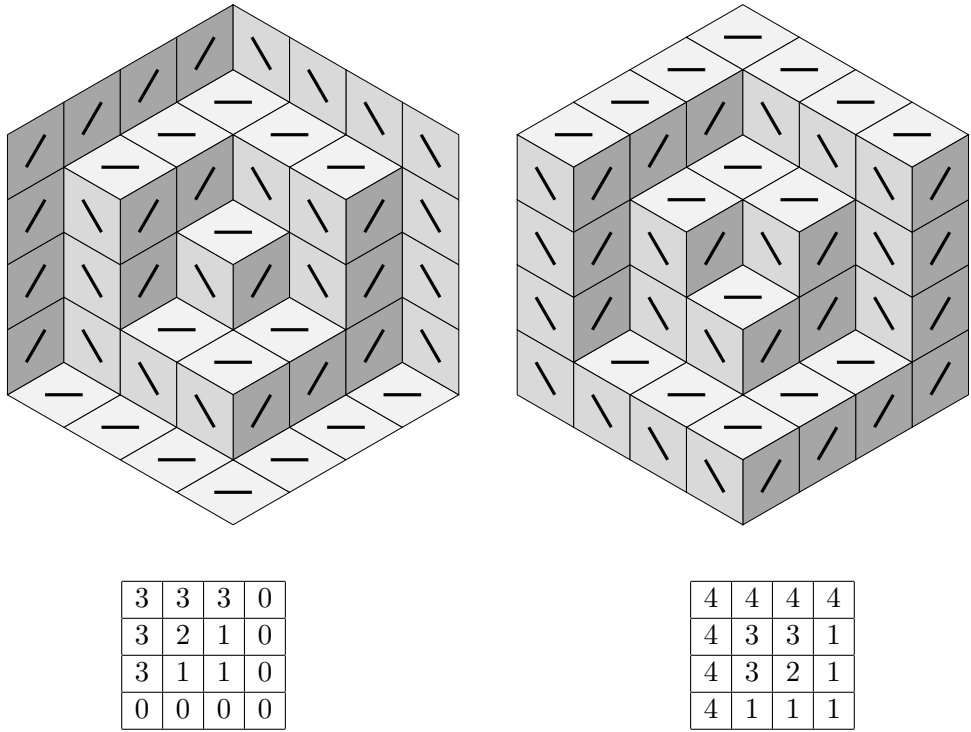
\begin{figure}[ht]
			$$\begin{tikzpicture}[scale = 0.25, x=1cm,y=1cm]
				\def \hex {+(0, 0) -- +(2, 0) -- +(3, 1*1.73) -- +(2, 2*1.73) -- +(0, 2*1.73) -- +(-1, 1*1.73) -- +(0, 0)};
				\def \up {+(0, 0) -- +(3, 1.73) -- +(3, -1*1.73) -- +(0, -2*1.73) -- +(0,0)
					+(1, -1*1.73) -- +(2, 0*1.73)};
				\def \down {+(0, 0) -- +(-3, 1.73) -- +(-3, -1*1.73) -- +(0, -2*1.73) -- +(0,0)
					+(-1, -1.73) -- +(-2, 0)};
				\def \flat {+(0, 0) -- +(3, 1.73) -- +(0, 2*1.73) -- +(-3, 1*1.73) -- +(0,0)
					+(-1, 1.73) -- +(1, 1.73)};
				
				\def\up#1{
					\begin{scope}[shift={#1}]
						\draw [very thin, fill=black!35!white] (0, 0) -- (3, 1.73) -- (3, -1*1.73) -- (0, -2*1.73) -- (0,0);
						\draw [very thick] (1, -1*1.73) -- (2, 0*1.73);
					\end{scope}
				}
				\def\down#1{
					\begin{scope}[shift={#1}]
						\draw [very thin, fill=black!15!white] (0, 0) -- (-3, 1.73) -- (-3, -1*1.73) -- (0, -2*1.73) -- (0,0);
						\draw [very thick] (-1, -1.73) -- (-2, 0);
					\end{scope}
				}
				\def\flat#1{
					\begin{scope}[shift={#1}]
						\draw [very thin, fill=black!5!white] (0, 0) -- (3, 1.73) -- (0, 2*1.73) -- (-3, 1*1.73) -- (0,0);
						\draw [very thick] (-1, 1.73) -- (1, 1.73);
					\end{scope}
				}
				
				\flat{(1, 1.73)}
				\flat{(7, 3*1.73)}
				\flat{(-5, 3*1.73)}
				\flat{(1, -3*1.73)}
				\flat{(1, -7*1.73)}
				\flat{(4, -6*1.73)}
				\flat{(7, -5*1.73)}
				\flat{(10, -4*1.73)}
				\flat{(-2, -6*1.73)}
				\flat{(-5, -5*1.73)}
				\flat{(-8, -4*1.73)}
				\flat{(1, 5*1.73)}
				\flat{(-2, 4*1.73)}
				\flat{(4, 4*1.73)}
				\flat{(-2, -2*1.73)}
				\flat{(4, -2*1.73)}
				\up{(1, 1.73)}
				\up{(7, 3*1.73)}
				\up{(-5, 3*1.73)}
				\up{(1, -3*1.73)}
				\up{(-5, 1*1.73)}
				\up{(4, -2*1.73)}
				\up{(7, -1*1.73)}
				\up{(-11, -1*1.73)}
				\up{(-11, 1*1.73)}
				\up{(-11, 3*1.73)}
				\up{(-11, 5*1.73)}
				\up{(-8, 6*1.73)}
				\up{(-5, 7*1.73)}
				\up{(-2, 8*1.73)}
				\up{(-2, 4*1.73)}
				\up{(7, 1*1.73)}
				\down{(4, 4*1.73)}
				\down{(-5, 1*1.73) }
				\down{(-5, -1*1.73)}
				\down{(-2, -2*1.73)}
				\down{(7, 1*1.73)}
				\down{(1, -3*1.73)} 
				\down{(-5, 3*1.73)}
				\down{(1, 1.73)} 
				\down{(7, 3*1.73)}
				\down{(13, -1*1.73)}
				\down{(13, 1*1.73)}
				\down{(13, 3*1.73)}
				\down{(13, 5*1.73)}
				\down{(10, 6*1.73)}
				\down{(7, 7*1.73)}
				\down{(4, 8*1.73)}
			\end{tikzpicture} 
			\qquad \begin{tikzpicture}[scale = 0.25, x=1cm,y=1cm]
				\def \hex {+(0, 0) -- +(2, 0) -- +(3, 1*1.73) -- +(2, 2*1.73) -- +(0, 2*1.73) -- +(-1, 1*1.73) -- +(0, 0)};
				\def \up {+(0, 0) -- +(3, 1.73) -- +(3, -1*1.73) -- +(0, -2*1.73) -- +(0,0)
					+(1, -1*1.73) -- +(2, 0*1.73)};
				\def \down {+(0, 0) -- +(-3, 1.73) -- +(-3, -1*1.73) -- +(0, -2*1.73) -- +(0,0)
					+(-1, -1.73) -- +(-2, 0)};
				\def \flat {+(0, 0) -- +(3, 1.73) -- +(0, 2*1.73) -- +(-3, 1*1.73) -- +(0,0)
					+(-1, 1.73) -- +(1, 1.73)};
				
				\def\up#1{
					\begin{scope}[shift={#1}]
						\draw [very thin, fill=black!35!white] (0, 0) -- (3, 1.73) -- (3, -1*1.73) -- (0, -2*1.73) -- (0,0);
						\draw [very thick] (1, -1*1.73) -- (2, 0*1.73);
					\end{scope}
				}
				\def\down#1{
					\begin{scope}[shift={#1}]
						\draw [very thin, fill=black!15!white] (0, 0) -- (-3, 1.73) -- (-3, -1*1.73) -- (0, -2*1.73) -- (0,0);
						\draw [very thick] (-1, -1.73) -- (-2, 0);
					\end{scope}
				}
				\def\flat#1{
					\begin{scope}[shift={#1}]
						\draw [very thin, fill=black!5!white] (0, 0) -- (3, 1.73) -- (0, 2*1.73) -- (-3, 1*1.73) -- (0,0);
						\draw [very thick] (-1, 1.73) -- (1, 1.73);
					\end{scope}
				}
				
				\flat{(-11, 5*1.73)}
				\flat{(-8, 6*1.73)}
				\flat{(-5, 7*1.73)}
				\flat{(-2, 8*1.73)}
				\flat{(1, 7*1.73)}
				\flat{(4, 6*1.73)}
				\flat{(7, 5*1.73)}
				\flat{(-2, -4*1.73)}
				\flat{(-5, -3*1.73)}
				\flat{(1, -3*1.73)}
				\flat{(4, -2*1.73)}
				\flat{(-8, -2*1.73)}
				\flat{(-2, 0*1.73)}
				\flat{(-2, 4*1.73)}
				\flat{(-5, 3*1.73)}
				\flat{(1, 3*1.73)}
				
				\up{(1, 1.73)}
				\up{(1, 3*1.73)}
				\up{(-2, 0*1.73)}
				\up{(-5, 3*1.73)}
				\up{(-11, 1*1.73)}
				\up{(-11, 3*1.73)}
				\up{(-11, 5*1.73)}
				\up{(-8, 6*1.73)}
				\up{(-5, 7*1.73)}
				\up{(7, 1*1.73)}
				\up{(7, 3*1.73)}
				\up{(7, 5*1.73)}
				\up{(7, -1*1.73)}
				\up{(4, -2*1.73)}
				\up{(1, -3*1.73)}
				\up{(-2, -4*1.73)}
				
				\down{(-5, 1*1.73) }
				\down{(-8, -2*1.73)}
				\down{(-5, -3*1.73)}
				\down{(-2, 0*1.73)}
				\down{(7, 1*1.73)}
				\down{(-2, -4*1.73)} 
				\down{(-5, 3*1.73)}
				\down{(7, 3*1.73)}
				\down{(7, 5*1.73)}
				\down{(4, 6*1.73)}
				\down{(1, 7*1.73)}
				\down{(1, 3*1.73)}
				\down{(-11, -1*1.73)}
				\down{(-11, 1*1.73)}
				\down{(-11, 3*1.73)}
				\down{(-11, 5*1.73)}
			\end{tikzpicture}$$
			
				$$\begin{array}{|c|c|c|c|}
					\hline
					3 & 3 & 3 & 0\\\hline
					3 & 2 & 1 & 0 \\\hline
					3 & 1 & 1 & 0 \\\hline
					0 & 0 & 0 & 0 \\\hline
				\end{array} \qquad\qquad\qquad\qquad\qquad\qquad\qquad \begin{array}{|c|c|c|c|}
					\hline
					4 & 4 & 4 & 4 \\\hline
					4 & 3 & 3 & 1 \\\hline
					4 & 3 & 2 & 1 \\\hline
					4 & 1 & 1 & 1 \\\hline
				\end{array}$$
			\caption{Two single dimer configurations shown as boxed plane partitions}
			\label{duploex}
		\end{figure}

	Consider the two plane partitions given by Figure \ref{duploex}. Both of these configurations become a loop within a loop (as in Figure \ref{looploop}) when sent through the squish map. If we picture the plane partitions as boxes in a room, the minimal configuration that squishes to the loop within a loop is the diagram on the left of Figure \ref{duploex}, and the maximal one is the diagram on the right. (Note that these are just two out of a possible 23,364 configurations that squish to the same loop-within-a-loop. See Theorem \ref{firstroot}.) To determine which possible two overlayed $2 \times 2 \times 2$ single dimer configurations give the same double dimer configuration, we start by downsampling either plane partition. 
	
	For this example, we will work through the process on the minimum diagram and partition above, though the process will work on either partition (or any of the other $23,362$ in between).
	 
	We want to round down to get the first single dimer configuration, $\pi_{min}$, and round up to get the second single dimer configuration, $\pi_{max}$. To do this rounding process, start by sectioning the plane partition into $2 \times 2$ grids. $$\begin{tikzpicture}[scale = 0.55]
		\draw[very thick] (0, 0) -- (0, 4);
		\draw (1, 0) -- (1, 4);
		\draw[very thick] (2, 0) -- (2, 4);
		\draw (3, 0) -- (3, 4);
		\draw[very thick] (4, 0) -- (4, 4);
		\draw[very thick] (0, 0) -- (4, 0);
		\draw (0, 1) -- (4, 1);
		\draw[very thick] (0, 2) -- (4, 2);
		\draw (0, 3) -- (4, 3);
		\draw[very thick] (0, 4) -- (4, 4);
		\node at (0.5, 0.5) {0};
		\node at (0.5, 1.5) {3};
		\node at (0.5, 2.5) {3};
		\node at (0.5, 3.5) {3};
		
		\node at (1.5, 0.5) {0};
		\node at (1.5, 1.5) {1};
		\node at (1.5, 2.5) {2};
		\node at (1.5, 3.5) {3};
		
		\node at (2.5, 0.5) {0};
		\node at (2.5, 1.5) {1};
		\node at (2.5, 2.5) {1};
		\node at (2.5, 3.5) {3};
		
		\node at (3.5, 0.5) {0};
		\node at (3.5, 1.5) {0};
		\node at (3.5, 2.5) {0};
		\node at (3.5, 3.5) {0};
		
	\end{tikzpicture}$$
	
	Each $2\times 2$ section will become one entry in the downsampled plane partition. When we round down, we want to count how many \emph{complete} $2 \times 2 \times 2$ blocks exist in each section. (You can also think about this as removing smaller $1 \times 1 \times 1$ cubes one-at-a-time until you are only left with $2\times 2\times 2$ blocks.) Thus, we have \raisebox{-0.48cm}{\begin{tikzpicture}[scale = 0.55]
		\draw[very thick] (0, 0) -- (0, 2);
		\draw (1, 0) -- (1, 2);
		\draw[very thick] (2, 0) -- (2, 2);
		\draw[very thick] (0, 0) -- (2, 0);
		\draw (0, 1) -- (2, 1);
		\draw[very thick] (0, 2) -- (2, 2);
		\node at (0.5, 1.5) {1};
		\node at (0.5, 0.5) {0};
		
		\node at (1.5, 1.5) {0};
		\node at (1.5, 0.5) {0};
	\end{tikzpicture}} as the first single dimer configuration. 
	
	Now to find the second single dimer configuration, we want to round up to the nearest larger cube. (You can also think about this as adding smaller cubes until we fill the $2 \times 2 \times 2$ cube.) So the maximal configuration in its entirety becomes rewritten as \raisebox{-0.48cm}{\begin{tikzpicture}[scale = 0.55]
		\draw[very thick] (0, 0) -- (0, 2);
		\draw (1, 0) -- (1, 2);
		\draw[very thick] (2, 0) -- (2, 2);
		\draw[very thick] (0, 0) -- (2, 0);
		\draw (0, 1) -- (2, 1);
		\draw[very thick] (0, 2) -- (2, 2);
		\node at (0.5, 1.5) {2};
		\node at (0.5, 0.5) {2};
		
		\node at (1.5, 1.5) {2};
		\node at (1.5, 0.5) {1};
	\end{tikzpicture}}.
	
	Now we have a pair of plane partitions that, when overlayed, become the same double dimer configuration we had as the result under the squish map. In Figure \ref{looploop}, the yellow (lighter colored) configuration corresponds to \raisebox{-0.48cm}{\begin{tikzpicture}[scale = 0.55]		\draw[very thick] (0, 0) -- (0, 2);
		\draw (1, 0) -- (1, 2);
		\draw[very thick] (2, 0) -- (2, 2);
		\draw[very thick] (0, 0) -- (2, 0);
		\draw (0, 1) -- (2, 1);
		\draw[very thick] (0, 2) -- (2, 2);
		\node at (0.5, 1.5) {1};
		\node at (0.5, 0.5) {0};
		
		\node at (1.5, 1.5) {0};
		\node at (1.5, 0.5) {0};
	\end{tikzpicture}}, and the blue (darker colored) configuration corresponds to \raisebox{-0.48cm}{\begin{tikzpicture}[scale = 0.55]
		\draw[very thick] (0, 0) -- (0, 2);
		\draw (1, 0) -- (1, 2);
		\draw[very thick] (2, 0) -- (2, 2);
		\draw[very thick] (0, 0) -- (2, 0);
		\draw (0, 1) -- (2, 1);
		\draw[very thick] (0, 2) -- (2, 2);
		\node at (0.5, 1.5) {2};
		\node at (0.5, 0.5) {2};
		
		\node at (1.5, 1.5) {2};
		\node at (1.5, 0.5) {1};
		\end{tikzpicture}}.
	
		\begin{figure}[h]
			$$\begin{tikzpicture}[x=1pt,y=1pt]
				\draw[line width = 2.5pt, yellow] (310.0, 601.7691453623979) -- (300.0, 584.4486372867091);
				\draw[line width = 2.5pt, yellow] (340.0, 515.166604983954) -- (360.0, 515.166604983954);
				\draw[line width = 2.5pt, yellow] (390.0, 601.7691453623979) -- (400.0, 584.4486372867091);
				\draw[line width = 2.5pt, yellow] (340.0, 584.4486372867091) -- (360.0, 584.4486372867091);
				\draw[line width = 2.5pt, yellow] (340.0, 619.0896534380867) -- (330.0, 601.7691453623979);
				\draw[line width = 2.5pt, yellow] (370.0, 601.7691453623979) -- (360.0, 619.0896534380867);
				\draw[line width = 2.5pt, yellow] (310.0, 567.1281292110203) -- (300.0, 549.8076211353316);
				\draw[line width = 2.5pt, yellow] (310.0, 532.4871130596428) -- (330.0, 532.4871130596428);
				\draw[line width = 2.5pt, yellow] (340.0, 549.8076211353316) -- (330.0, 567.1281292110203);
				\draw[line width = 2.5pt, yellow] (360.0, 549.8076211353316) -- (370.0, 567.1281292110203);
				\draw[line width = 2.5pt, yellow] (370.0, 532.4871130596428) -- (390.0, 532.4871130596428);
				\draw[line width = 2.5pt, yellow] (400.0, 549.8076211353316) -- (390.0, 567.1281292110203);
				
				\draw[line width = 2.5pt, blue] (310.0, 601.7691453623979) -- (330.0, 601.7691453623979);
				\draw[line width = 2.5pt, blue] (370.0, 601.7691453623979) -- (390.0, 601.7691453623979);
				\draw[line width = 2.5pt, blue] (340.0, 619.0896534380867) -- (360.0, 619.0896534380867);
				\draw[line width = 2.5pt, blue] (340.0, 515.166604983954) -- (330.0, 532.4871130596428);
				\draw[line width = 2.5pt, blue] (400.0, 584.4486372867091) -- (390.0, 567.1281292110203);
				\draw[line width = 2.5pt, blue] (370.0, 532.4871130596428) -- (360.0, 515.166604983954);
				\draw[line width = 2.5pt, blue] (340.0, 549.8076211353316) -- (360.0, 549.8076211353316);
				\draw[line width = 2.5pt, blue] (340.0, 584.4486372867091) -- (330.0, 567.1281292110203);
				\draw[line width = 2.5pt, blue] (370.0, 567.1281292110203) -- (360.0, 584.4486372867091);
				\draw[line width = 2.5pt, blue] (310.0, 567.1281292110203) -- (300.0, 584.4486372867091);
				\draw[line width = 2.5pt, blue] (310.0, 532.4871130596428) -- (300.0, 549.8076211353316);
				\draw[line width = 2.5pt, blue] (400.0, 549.8076211353316) -- (390.0, 532.4871130596428);
				
				\draw (310.0, 532.4871130596428) -- (300.0, 549.8076211353316);
				\draw (310.0, 567.1281292110203) -- (300.0, 549.8076211353316);
				\draw (310.0, 567.1281292110203) -- (300.0, 584.4486372867091);
				\draw (310.0, 601.7691453623979) -- (300.0, 584.4486372867091);
				
				\draw (340.0, 515.166604983954) -- (360.0, 515.166604983954);
				\draw (340.0, 515.166604983954) -- (330.0, 532.4871130596428);
				\draw (310.0, 532.4871130596428) -- (330.0, 532.4871130596428);
				\draw (340.0, 549.8076211353316) -- (330.0, 532.4871130596428);
				\draw (340.0, 549.8076211353316) -- (360.0, 549.8076211353316);
				\draw (340.0, 549.8076211353316) -- (330.0, 567.1281292110203);
				\draw (310.0, 567.1281292110203) -- (330.0, 567.1281292110203);
				\draw (340.0, 584.4486372867091) -- (330.0, 567.1281292110203);
				\draw (340.0, 584.4486372867091) -- (360.0, 584.4486372867091);
				\draw (340.0, 584.4486372867091) -- (330.0, 601.7691453623979);
				\draw (310.0, 601.7691453623979) -- (330.0, 601.7691453623979);
				\draw (340.0, 619.0896534380867) -- (330.0, 601.7691453623979);
				\draw (340.0, 619.0896534380867) -- (360.0, 619.0896534380867);

				\draw (370.0, 532.4871130596428) -- (360.0, 515.166604983954);
				\draw (370.0, 532.4871130596428) -- (360.0, 549.8076211353316);
				\draw (370.0, 567.1281292110203) -- (360.0, 549.8076211353316);
				\draw (370.0, 567.1281292110203) -- (360.0, 584.4486372867091);
				\draw (370.0, 601.7691453623979) -- (360.0, 584.4486372867091);
				\draw (370.0, 601.7691453623979) -- (360.0, 619.0896534380867);
				
				\draw (370.0, 532.4871130596428) -- (390.0, 532.4871130596428);
				\draw (400.0, 549.8076211353316) -- (390.0, 532.4871130596428);
				\draw (400.0, 549.8076211353316) -- (390.0, 567.1281292110203);
				\draw (370.0, 567.1281292110203) -- (390.0, 567.1281292110203);
				\draw (400.0, 584.4486372867091) -- (390.0, 567.1281292110203);
				\draw (390.0, 601.7691453623979) -- (400.0, 584.4486372867091);
				\draw (370.0, 601.7691453623979) -- (390.0, 601.7691453623979);\end{tikzpicture}$$
			
			\caption{A loop within a loop in the double dimer model}
			\label{looploop}
		\end{figure}
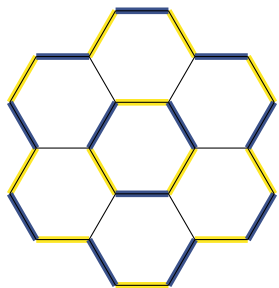
	\end{example}

	\begin{example}\label{example:downsample}
		For the plane partition $$\begin{tikzpicture}[scale = 0.5]
			\draw[very thick] (0, 0) -- (0, 8);
			\draw (1, 0) -- (1, 8);
			\draw[very thick] (2, 0) -- (2, 8);
			\draw (3, 0) -- (3, 8);
			\draw[very thick] (4, 0) -- (4, 8);
			\draw[very thick] (0, 0) -- (4, 0);
			\draw (0, 1) -- (4, 1);
			\draw[very thick] (0, 2) -- (4, 2);
			\draw (0, 3) -- (4, 3);
			\draw[very thick] (0, 4) -- (4, 4);
			\draw (0, 5) -- (4, 5);
			\draw[very thick] (0, 6) -- (4, 6);
			\draw (0, 7) -- (4, 7);
			\draw[very thick] (0, 8) -- (4, 8);
			
			\node at (0.5, 0.5) {1};
			\node at (0.5, 1.5) {2};
			\node at (0.5, 2.5) {3};
			\node at (0.5, 3.5) {4};
			\node at (0.5, 4.5) {5};
			\node at (0.5, 5.5) {6};
			\node at (0.5, 6.5) {7};
			\node at (0.5, 7.5) {8};
			
			\node at (1.5, 0.5) {1};
			\node at (1.5, 1.5) {2};
			\node at (1.5, 2.5) {3};
			\node at (1.5, 3.5) {3};
			\node at (1.5, 4.5) {4};
			\node at (1.5, 5.5) {4};
			\node at (1.5, 6.5) {6};
			\node at (1.5, 7.5) {8};
			
			\node at (2.5, 0.5) {1};
			\node at (2.5, 1.5) {1};
			\node at (2.5, 2.5) {2};
			\node at (2.5, 3.5) {3};
			\node at (2.5, 4.5) {3};
			\node at (2.5, 5.5) {3};
			\node at (2.5, 6.5) {6};
			\node at (2.5, 7.5) {6};
			
			\node at (3.5, 0.5) {0};
			\node at (3.5, 1.5) {1};
			\node at (3.5, 2.5) {1};
			\node at (3.5, 3.5) {2};
			\node at (3.5, 4.5) {3};
			\node at (3.5, 5.5) {3};
			\node at (3.5, 6.5) {5};
			\node at (3.5, 7.5) {5};
			
			\end{tikzpicture}$$
		then $$\begin{tikzpicture}[scale = 0.5]
			\node at (-2, 4) {$\pi_{min} = $};
			\draw[very thick] (0, 0) -- (0, 8);
			\draw[very thick] (2, 0) -- (2, 8);
			\draw[very thick] (4, 0) -- (4, 8);
			\draw[very thick] (0, 0) -- (4, 0);
			\draw[very thick] (0, 2) -- (4, 2);
			\draw[very thick] (0, 4) -- (4, 4);
			\draw[very thick] (0, 6) -- (4, 6);
			\draw[very thick] (0, 8) -- (4, 8);
			
			\node at (1, 1) {\Large 0};
			\node at (1, 3) {\Large 1};
			\node at (1, 5) {\Large 2};
			\node at (1, 7) {\Large 3};
			
			\node at (3, 1) {\Large 0};
			\node at (3, 3) {\Large 0};
			\node at (3, 5) {\Large 1};
			\node at (3, 7) {\Large 2};
			\node at (7, 4) {and};
		\end{tikzpicture} \qquad \begin{tikzpicture}[scale = 0.5]
		\node at (-2, 4) {$\pi_{max} = $};
		\draw[very thick] (0, 0) -- (0, 8);
		\draw[very thick] (2, 0) -- (2, 8);
		\draw[very thick] (4, 0) -- (4, 8);
		\draw[very thick] (0, 0) -- (4, 0);
		\draw[very thick] (0, 2) -- (4, 2);
		\draw[very thick] (0, 4) -- (4, 4);
		\draw[very thick] (0, 6) -- (4, 6);
		\draw[very thick] (0, 8) -- (4, 8);
		
		\node at (1, 1) {\Large 1};
		\node at (1, 3) {\Large 2};
		\node at (1, 5) {\Large 3};
		\node at (1, 7) {\Large 4};
		
		\node at (3, 1) {\Large 1};
		\node at (3, 3) {\Large 2};
		\node at (3, 5) {\Large 2};
		\node at (3, 7) {\Large 3};
		\end{tikzpicture}$$
	\end{example}
	
	\begin{remark}
		Consider a plane partition that could already have been made out of $2 \times 2 \times 2$ boxes. We can downsample in a straightforward manner. Each $2 \times 2$ region in the plane partition contains all the same even entry, so the entry in $\pi_{min}$ would be half that number, as it would be for $\pi_{max}$. Since all four entries were the same, then we have no rounding to do, so $\pi_{min} = \pi_{max}$. Thus when we overlay the single dimer configurations corresponding to $\pi_{min}$ and $\pi_{max}$, we get all doubled edges in the double dimer configuration. 
	\end{remark}
	
	\begin{proof}[Proof of Theorem \ref{duplo}]
		
		Consider a given plane partition $\pi$ that may not break down nicely into $2 \times 2 \times 2$ cubes, where 
		$\pi$ has $k$ boxes. Then $Sq(\pi) = \pi_{min} \sqcup \pi_{max}$ as double dimer configurations. We proceed by induction on $k$.  The base case is $k=0$, which is straightforward by the remarks above: $\pi, \pi_{\min}, \pi_{\max}$ and all the corresponding single  or double  dimer configurations are minimal.
		
		Suppose that $\pi$ has $k + 1$ boxes. Delete a box from $\pi$ to create a $\pi'$ that has $k$ boxes. Then via the induction hypothesis we have that $Sq(\pi') = \pi'_{min} \sqcup \pi'_{max}$. 
		If we add the $(k+1)$th box back in, we land in one of three cases. 
		
		In the first case, we have started a new $2 \times 2 \times 2$ cube by adding that cube in the $(2i, 2j, 2k)$ position. Then, when viewed as a plane partition,
		$\pi_{max}(i,j) = \pi'_{max}(i,j) + 1$, while the minimum configuration stays the same
		($\pi_{min}(i,j) = \pi'_{min}(i,j)$). Under the squish map we get that $Sq(\pi) = Sq(\pi')$, except at the hexagon located at $(i,j,k)$, where the perfect matching around said hexagon changes from the left of Figure \ref{boxchange} to the right. 
		
		In the second case, we have added the cube into the last empty slot of a $2 \times 2 \times 2$ larger cube, so the new cube has gone into position $(2{i+1}, 2{j+1}, 2{k+1})$. Then the maximum plane partition remains the same, so $\pi_{max}(i,j) = \pi'_{max}(i,j)$, but $\pi_{min}(i,j) = \pi'_{min}(i, j)+1$. So under the squish map we have that $Sq(\pi) = Sq(\pi')$, except at $(i,j,k)$, where the matching around the hexagon changes from the right of Figure \ref{boxchange} to the left. 
		
		Finally, the last case occurs when we add a cube anywhere else; i.e. adding this smaller cube neither completes a $2 \times 2 \times 2$ box nor is the first cube in an otherwise empty $2 \times 2 \times 2$ box. Here we have that $\pi_{min} = \pi'_{min}$ 
		and $\pi_{max} = \pi'_{max}$. In this case we have no change under the squish map, so $Sq(\pi) = Sq(\pi')$ exactly.
	\end{proof}
	
	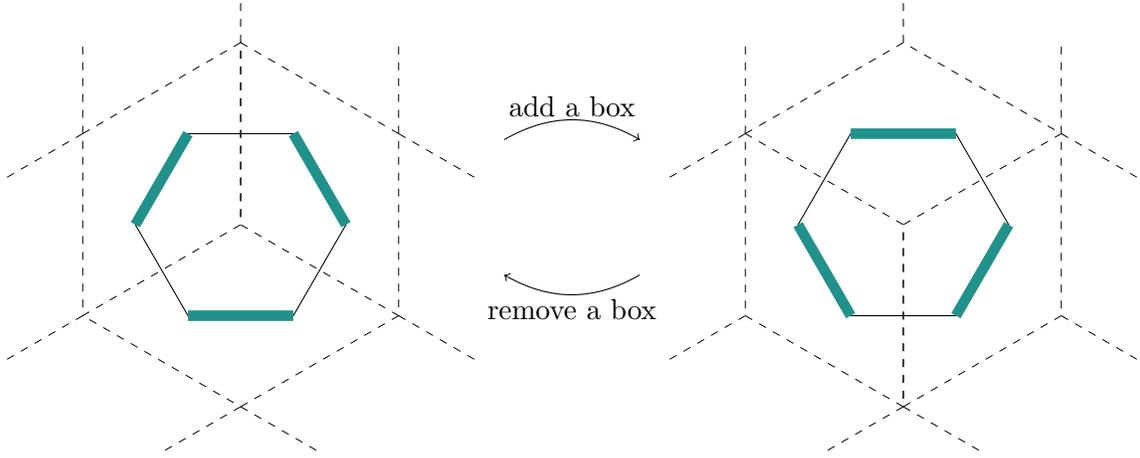
\begin{figure}[ht] 
		$$\begin{tikzpicture}[scale = 0.7]
			\draw [very thin, dashed] (5, -1*1.73) -- (2, -2*1.73) -- (-1, -1*1.73);
			\draw [very thin, dashed] (5, 1*1.73) -- (2, 2*1.73) -- (2, 0*1.73) -- (5, -1*1.73) -- (5,1*1.73);
			\draw [very thin, dashed] (-1, 1*1.73) -- (2, 2*1.73) -- (2, 0*1.73) -- (-1, -1*1.73) -- (-1,1*1.73);
			\draw[thin] (0, 0) -- (1, -1.73) -- (3, -1.73) -- (4, 0) -- (3, 1.73) -- (1, 1.73) -- (0, 0);
			\draw[line width = 4pt, color = green] (0, 0) (1, -1.73) -- (3, -1.73) (4, 0) -- (3, 1.73) (1, 1.73) -- (0, 0);
			\draw[very thin, dashed] (-1, 1.73) -- (-1, 2*1.73);
			\draw[very thin, dashed] (2, 2*1.73) -- (2, 2*1.73+1.73/2);
			\draw[very thin, dashed] (5, 1.73) -- (5, 2*1.73);
			\draw[very thin, dashed] (-1, 1.73) -- (-2.5, 1.73/2);
			\draw[very thin, dashed] (-1, -1.73) -- (-2.5, -3*1.73/2);
			\draw[very thin, dashed] (5, 1.73) -- (6.5, 1.73/2);
			\draw[very thin, dashed] (5, -1.73) -- (6.5, -3*1.73/2);
			\draw[very thin, dashed] (2, -2*1.73) -- (0.5, -5*1.73/2); 
			\draw[very thin, dashed] (2, -2*1.73) -- (3.5, -5*1.73/2); 
		\end{tikzpicture}
		\begin{tikzpicture}[scale = 0.9]
			\draw[->] (4, 1) to[bend left] (6, 1);
			\draw[->] (6, -1) to[bend left] (4, -1);
			\node[anchor = center] at (5, 1.5) {add a box};
			\node[anchor = center] at (5, -1.5) {remove a box};
			\node[color = white] at (5, -2*1.73) {.};
		\end{tikzpicture}
		\begin{tikzpicture}[scale = 0.7]
			\draw [very thin, dashed] (5, -1*1.73) -- (8, -2*1.73) -- (8, 0*1.73) -- (5, 1*1.73) -- (5, -1*1.73);
			\draw [very thin, dashed] (11, -1*1.73) -- (8, -2*1.73) -- (8, 0*1.73) -- (11, 1*1.73) -- (11,-1*1.73);
			\draw [very thin, dashed] (5, 1.73) -- (8, 2*1.73) -- (11, 1.73);
			\draw[thin] (6, 0) -- (7, -1.73) -- (9, -1.73) -- (10, 0) -- (9, 1.73) -- (7, 1.73) -- (6, 0); 				
			\draw[line width = 4pt, color = green] (6, 0) -- (7, -1.73) (9, -1.73) -- (10, 0) (9, 1.73) -- (7, 1.73) (6, 0);
			\draw[very thin, dashed] (-1+6, 1.73) -- (-1+6, 2*1.73);
			\draw[very thin, dashed] (2+6, 2*1.73) -- (2+6, 2*1.73+1.73/2);
			\draw[very thin, dashed] (5+6, 1.73) -- (5+6, 2*1.73);
			\draw[very thin, dashed] (-1+6, 1.73) -- (-2.5+6, 1.73/2);
			\draw[very thin, dashed] (-1+6, -1.73) -- (-2.5+6, -3*1.73/2);
			\draw[very thin, dashed] (5+6, 1.73) -- (6.5+6, 1.73/2);
			\draw[very thin, dashed] (5+6, -1.73) -- (6.5+6, -3*1.73/2);
			\draw[very thin, dashed] (2+6, -2*1.73) -- (0.5+6, -5*1.73/2); 
			\draw[very thin, dashed] (2+6, -2*1.73) -- (3.5+6, -5*1.73/2); 
		\end{tikzpicture}$$
		
		\caption{Adding or removing a box}
		\label{boxchange}
	\end{figure}
	
	\subsection{Transfer matrix approach}
            Now that we have the squish map defined on graphs and plane partitions, we want to be able to take a weight function in the single dimer model and push it through the squish map to give us an $\SL_2$ connection and a scalar weight function in the double dimer model. To do this, we first consider a perfect matching on the single dimer model. Once the graph has been squished, we can consider walking along a path around a given loop in the now-double dimer model. Given a starting vertex, each path consists of a series of right and left turns at each new vertex encountered until we once again reach the starting vertex. These turns are given labels $L$ and $R$ for left and right assigned to each vertex in the path.  We want to interpret $L$ and $R$ as $2 \times 2$ transfer matrices for keeping track of the loop's contribution to the dimer model. This was the strategy of \cite{young:squish}, which used different matrices. 
            
            To determine what particular matrices $L$ and $R$ should be, we consider a walk along the path snippet given by Figure \ref{leftturn}. We begin at the bottom edge and then turn left at the next vertex, so we could step onto either edge $z$ or $w$. If $x$ is an edge in the perfect matching, then $z$ cannot also be a matched edge, so the only next step could be $w$. Similarly, if $y$ is in the perfect matching, then a left turn onto the next matched edge could include either $z$ or $w$. 
            
            \begin{figure}[ht]
            $$\begin{tikzpicture}
            		\node at (-8, 0) {
            			$$\begin{tikzpicture}[scale = 3.5]
            				\def \squish {+(0: 1cm) -- 
            					++(60: 1cm) -- +(60: .155cm) -- +(240: 0cm) -- 
            					++(180:1cm) -- +(120: .155cm) -- +(240: 0cm) -- 
            					++(240: 1cm) -- +(180: .155cm) -- + (240: 0cm) -- 
            					++(300: 1cm) -- +(240: .155cm) -- +(240: 0cm) -- 
            					++(360: 1cm) -- +(300: .155cm) -- +(240: 0cm) -- 
            					++(420: 1cm) -- +(360: .155cm) -- +(240: 0cm)}
            				\clip[draw] (1.3,-.3) rectangle (2.93,.9);
            				\node at (1.8, .2) {$x$};
            				\node at (1.8, -.2) {$y$};
            				\node at (2.4, .7) {$z$};
            				\node at (2.8, .4) {$w$};
            				
            				\draw (0,0) \squish;
            				\draw (1.73, 1) \squish;
            				\draw (1.73, -1) \squish;
            				\draw (1.73*2, 0) \squish;
            			\end{tikzpicture}$$ 
            			
            		};
            		\clip[draw] (0,0) circle (3);
            		\node at (0,0) {\includegraphics[width = .5\textwidth]{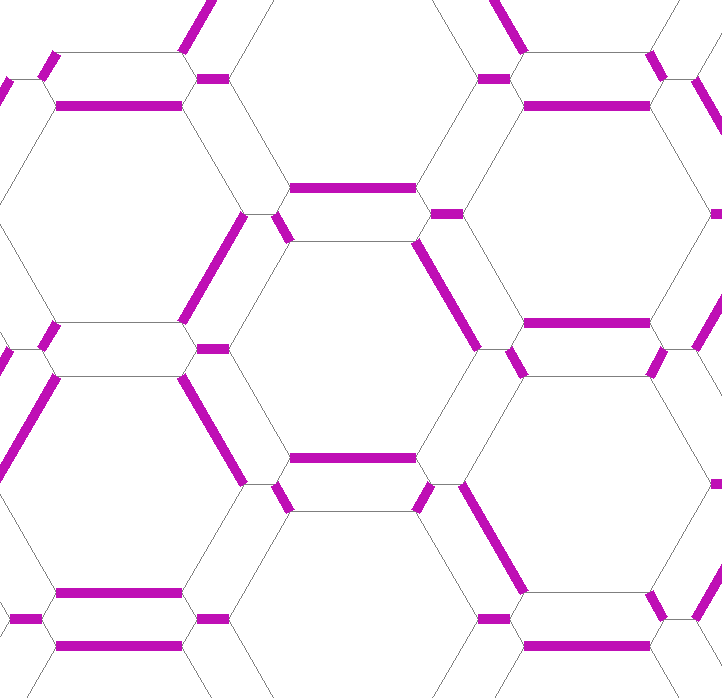}};
            	\end{tikzpicture}$$
            \caption{Left: the possible $x, y, z$, and $w$ edge weights over a left-hand turn.
            Right: A perfect matching on a squished hexagon lattice.}
            \label{leftturn}
            \end{figure}
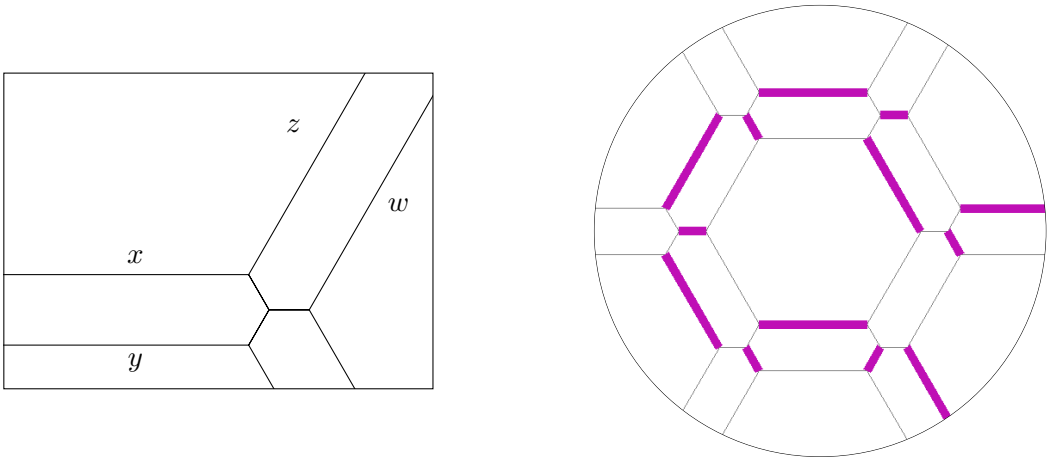
            
            We represent this with the vectors $$ \begin{bmatrix} x \\ 0 \end{bmatrix} \qquad \begin{bmatrix} 0 \\ y\end{bmatrix} $$
            
            as our two possible starting locations. Then we need the first vector to map only to edge $w$, and the second vector to map to both edges $z$ and $w$, so we get the following maps 
             $$\begin{bmatrix}x \\ 0 \end{bmatrix} \mapsto \begin{bmatrix} 0 \\ w\end{bmatrix} \qquad \begin{bmatrix}0 \\ y \end{bmatrix} \mapsto \begin{bmatrix} z \\ w\end{bmatrix}.$$

             A similar scenario happens for $R$, so we then find the $2 \times 2$ matrix to make the above maps hold, getting that
              $$L = \begin{bmatrix}
                    0 & 1 \\ 1 & 1
            \end{bmatrix} \qquad \text{ and } \qquad R = \begin{bmatrix}
                    1 & 1 \\ 1 & 0
            \end{bmatrix}$$

            \begin{remark}
            	Note that instead of defining $L$ and $R$ as above, we could also have defined $L$ to be $R^{-1}$, and similar for $R$. So instead we would have had $$L = \begin{bmatrix}
            		0 & 1 \\ 1 & -1
            	\end{bmatrix} \qquad \text{ and } \qquad R = \begin{bmatrix}
            		-1 & 1 \\ 1 & 0
            	\end{bmatrix}.$$ Sometimes it may be more convenient to define $L$ and $R$ this way (with the negative signs), such as if we had specialized the edge weights to be $\pm 1$. We ultimately wanted to use the version with only positive entries, however, to (hopefully) make it more clear to the reader how things are working.
                    \end{remark}
                    
           \subsection{From transfer matrices to the $\text{SL}_2$ connection}
We can now use the matrices $L$ and $R$ to compute the contribution of a closed loop under the squish map from the single to the double dimer model by taking the trace of the product. However, 
we're supposed to have $2 \times 2$ matrices, with determinant 1  associated to the \emph{edges} of $H(i,j,k)$, not the vertices. We use the following process to include the left and right turn information in the edges of the graph as matrix weights. 

To begin with, we associate general $2 \times 2$ edge-weight matrices $A, B$, and $C$ to be placed on each horizontal, north-east, and north-west matched edge, respectively, as in Figure \ref{oldweights}. 
\[
A =\left[\begin{matrix}a & 0 \\ 0 & a^{-1} \end{matrix} \right], \qquad
B=\left[\begin{matrix}b^{-1} & 0 \\ 0 & b \end{matrix} \right], \qquad
C=\left[\begin{matrix}c & 0 \\ 0 & c^{-1} \end{matrix} \right] 
\]

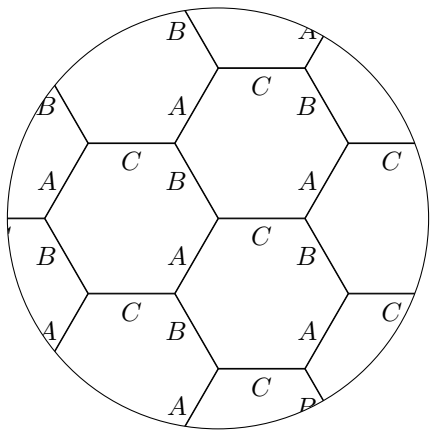
\begin{figure}[ht]	
	
	\small
	$$\begin{tikzpicture}[scale = 2]
		\def\hexagon{+(0:0.575cm) -- +(60:0.575cm) -- +(120:0.575cm) -- +(180:0.575cm) -- +(240:0.575cm) -- +(300:0.575cm) -- cycle}
		\def\hexagonweights{+(0:0.575cm) -- node [anchor = east]{$B$} +(60:0.575cm) -- node [anchor = north]{} +(120:0.575cm) -- node [anchor = west]{} +(180:0.575cm) -- node [anchor = west]{} +(240:0.575cm) -- node [anchor = north]{$C$} +(300:0.575cm) -- node [anchor = east]{$A$}cycle;}
		\def \squish {+(0: 1cm) -- 
			++(60: 1cm) -- +(60: .155cm) -- +(240: 0cm) -- 
			++(180:1cm) -- +(120: .155cm) -- +(240: 0cm) -- 
			++(240: 1cm) -- +(180: .155cm) -- + (240: 0cm) -- 
			++(300: 1cm) -- +(240: .155cm) -- +(240: 0cm) -- 
			++(360: 1cm) -- +(300: .155cm) -- +(240: 0cm) -- 
			++(420: 1cm) -- +(360: .155cm) -- +(240: 0cm)}
		
		\clip[draw] (0.575,1) circle (1.4cm);

		\foreach \x in {-5,...,5}
		\foreach \y\ in {-5,...,5}
		\draw (1.73*\x,\y) \hexagonweights;

		\foreach \x in {-5,...,5}
		\foreach \y\ in {-5,...,5}
		\draw (1.73*\x + 1.73*0.5,\y + 0.5) \hexagonweights;

	\end{tikzpicture}$$
	
	\normalsize 
	
	\caption{Matrix weights $A$, $B$, and $C$ on the single dimer model.}
	\label{oldweights}
\end{figure}

    To handle the process of moving $L$ and $R$ from the vertices to the edges, we need some way to include the information for turns in an unknown path. Then we want to come up with new matrices $\alpha$, $\beta$, and $\gamma$ that include the information from $A$, $B$, and $C$, but that also encode the information in the turns. For example, the move $A \to B$ is always a left turn, and the move $A \to C$ is always a right turn. So we somehow want the new $\beta\alpha$ to include the same information as $BLA$ and $\gamma\alpha$ to encode the information from $C R A$ (reading paths from right-to-left). 
    
    \begin{definition}
          	Let $J = \begin{bmatrix} 1 & 0 \\ 0 & -1\end{bmatrix}$, and $i = \sqrt{-1}$. 
          	
          	Then $\alpha$, $\beta$, and $\gamma$ are defined as follows:
          	
          	$$\displaystyle\alpha := iAJ = \begin{bmatrix}i a & 0 \vspace{1ex}\\
          		0 & \frac{1}{i a} \end{bmatrix}$$
          	
          	$$\beta := iRBLJ = 
          	\displaystyle\begin{bmatrix} i \, b & -i \, b - \frac{i}{b} \\
          		0 & \frac{1}{ib}\vspace{1ex} \end{bmatrix} 
          	$$
          	
          	$$\gamma := -iLCRJ  = 
          	\displaystyle\begin{bmatrix} \frac{1}{ic} & 0 \\
          		-i \, c - \frac{i}{c} & i \, c \end{bmatrix} .
          	$$
          	\label{matrices}
    \end{definition}
    
    \begin{theorem}
          	The single dimer model on the $2x \times 2y \times 2z$ hexagon lattice with $A$, $B$, and $C$ weights (as in Figure \ref{oldweights}) gives rise to the same partition function as that on the double dimer model on the $x \times y \times z$ hexagon graph with scalar weight of 1 everywhere and connection given by $\alpha$, $\beta$, and $\gamma$ (as in Figure \ref{newweights}). 
    \end{theorem}
    
    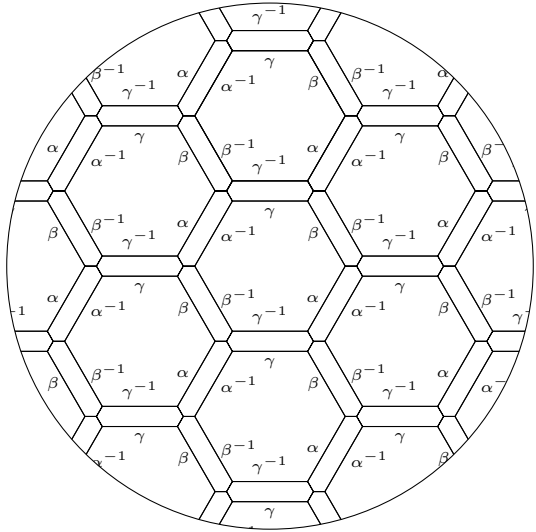
\begin{figure}[ht]
     \tiny
     $$\begin{tikzpicture}[scale = 1]
     	\def\hexagon{+(0:1cm) -- +(60:1cm) -- +(120:1cm) -- +(180:1cm) -- +(240:1cm) -- +(300:1cm) -- cycle}
     	\def\hexagonweights{+(0:1cm) -- node [anchor = east]{$\beta$} +(60:1cm) -- node [anchor = north]{$\gamma$} +(120:1cm) -- node [anchor = west]{$\alpha^{-1}$} +(180:1cm) -- node [anchor = west]{$\beta^{-1}$} +(240:1cm) -- node [anchor = south]{$\gamma^{-1}$} +(300:1cm) -- node [anchor = east]{$\alpha$}cycle;}
     	\def \squish {+(0: 1cm) -- 
     		++(60: 1cm) -- +(60: .155cm) -- +(240: 0cm) -- 
     		++(180:1cm) -- +(120: .155cm) -- +(240: 0cm) -- 
     		++(240: 1cm) -- +(180: .155cm) -- + (240: 0cm) -- 
     		++(300: 1cm) -- +(240: .155cm) -- +(240: 0cm) -- 
     		++(360: 1cm) -- +(300: .155cm) -- +(240: 0cm) -- 
     		++(420: 1cm) -- +(360: .155cm) -- +(240: 0cm)}
     	
     	\clip[draw] (1.73,1) circle (3.5cm);
     	\draw[=] (1.73, 1) \squish;
     	\draw[] (3.46, 2) \squish;
     	\draw (1.73, 3) \squish;
     	\draw (1.73, -1) \squish;
     	\draw (0,0) \squish;
     	\draw (1.73, 3) \squish;
     	\draw (0, 2) \squish;
     	\draw (-1.73, 1) \squish;
     	\draw(3.46, 0) \squish;
     	\draw(0, -2) \squish;
     	\draw(0,4) \squish;
     	\draw(-1.73, 3) \squish;
     	\draw(-1.73, -1) \squish;
     	\draw(1.73, 5) \squish;
     	\draw(1.73, -3) \squish;
     	\draw (3.46, 4) \squish;
     	\draw (3.46, -2) \squish;
     	\draw (3.46 + 1.73, 1) \squish;
     	\draw (3.46 + 1.73, 3) \squish;
     	\draw (3.46 + 1.73, -1) \squish;
     	
     	\foreach \x in {-5,...,5}
     	\foreach \y\ in {-5,...,5}{%
     		\draw [] (2 * 1.73*\x,2 * \y) \hexagonweights;
     	}

     	\foreach \x in {-5,...,5}
     	\foreach \y\ in {-5,...,5}{%
     		\draw[] (2 * 1.73*\x + 1.73,2 * \y + 1) \hexagonweights;
     	}
     	
     \end{tikzpicture}$$
     \normalsize
     
     \caption{The $\alpha$, $\beta$, and $\gamma$ connection on the squished hexagon lattice}
     \label{newweights}
    \end{figure}

\begin{proof}
	From the examples above, Let's consider the right turn $A \to C$. In our matrices, we write this as $\gamma\alpha$, which is $(-iLCRJ)(iAJ) = LCRA$, which does indeed include the information $CRA$, as we desired. There are 12 total ways to step from one edge to another, whose equivalent products are given below. 
	\begin{align*}
		\text{Left } & \text{Turns} & 
		\text{Right } & \text{Turns} \\
		\beta\alpha &= -RBLA & 
		\alpha\beta &= -JARBLJ \\
		\gamma\beta &= LCLBLJ & 
		\beta\gamma &= -RBRCRJ \\
		\alpha^{-1}\gamma &= -JA^{-1}LCRJ & 
		\gamma\alpha^{-1} &= -LCRA^{-1} \\
		\beta^{-1}\alpha^{-1} &= -RB^{-1}LA^{-1} & 
		\alpha^{-1}\beta^{-1}&= -JA^{-1}RB^{-1}LJ \\
		\gamma^{-1}\beta^{-1} &= LC^{-1}LB^{-1}LJ & 
		\beta^{-1}\gamma^{-1}&= RB^{-1}RC^{-1}RJ \\
		\alpha\gamma^{-1} &= -JALC^{-1}RJ & 
		\gamma^{-1}\alpha &= -LC^{-1}RA
	\end{align*}
	
    Assume we have a valid path of length $n$ with matrix string $\chi_n\chi_{n-1}\ldots \chi_1$. Then this string is equivalent to one of the form $$X_nM_nT_{n-1}M_{n-1}\ldots T_2M_2T_1M_1X_1,$$ where each $M \in \{A, B, C, A^{-1}, B^{-1}, C^{-1}\}$ is one of the original edge matrices, $T_i \in \{R, L\}$ is a turning matrix $L$ or $R$, and the $X_i$-terms on either end are the extras above used for bookkeeping (the $LJ$, $RJ$, $L$, $R$, or $J$ `bookkeeping' terms on the left or right of each full string). If we were to take another step along this path, then that $(n+1)$th step would involve left-multiplying our string by $\chi_{n+1}$, so we now have $\chi_{n+1}\chi_n\ldots \chi_1 = (\chi_{n+1}\chi_n)\chi_{n-1}\ldots \chi_1$. But the $\chi_{n+1}\chi_n$ is a valid construction of two edge-weight matrices with a left or right turn in the middle (possibly with an $X_{n+1}$ or $X_n$ bookkeeping term), which means it is one of the 12 turns explicitly computed above. So $\chi_{n+1}\chi_n\ldots \chi_1$ gives us a string that is equivalent to using $A$, $B$, and $C$ with left- and right-turn matrices. 
\end{proof}

    \begin{example}
            Consider the path around a single hexagon with these new weights. If we begin at the bottom edge and then travel counterclockwise around the loop, then we perform matrix multiplication (reading right to left) in the order we reach the edges, getting $\gamma^{-1}\beta^{-1}\alpha^{-1}\gamma\beta\alpha$. This product gives us the total connection for the path. 
            
            Now if we compute this matrix out for the left-turn version, we get
            \begin{align*}
            	&\hspace{-4ex}\gamma^{-1}\beta^{-1}\alpha^{-1}\gamma\beta\alpha\\
             \hspace{3ex}&= (iLC^{-1}RJ)(-iRB^{-1}LJ)(iA^{-1}J)(-iLCRJ)(iRBLJ)(iAJ)\\
             &= i^6LC^{-1}RJRB^{-1}LJA^{-1}JLCRJRBLJAJ\\
             &= -LC^{-1}LB^{-1}LA^{-1}iLCLBLA\\
            &= \begin{bmatrix}
            	a^{2} b^{2} c^{2} + a^{2} b^{2} + a^{2} c^{2} + a^{2} + 1 & \frac{{\left(a^{2} b^{2} c^{2} + a^{2} b^{2} + a^{2} + 1\right)} {\left(b^{2} + 1\right)}}{a^{2} b^{2}} \\
            	\frac{{\left(a^{2} b^{2} c^{2} + a^{2} c^{2} + a^{2} + 1\right)} {\left(c^{2} + 1\right)}}{c^{2}} & \frac{a^{2} b^{4} c^{4} + a^{2} b^{4} c^{2} + a^{2} b^{2} c^{4} + 3 \, a^{2} b^{2} c^{2} + a^{2} b^{2} + a^{2} c^{2} + b^{2} c^{2} + a^{2} + b^{2} + c^{2} + 1}{a^{2} b^{2} c^{2}}
            	\end{bmatrix}
			\end{align*}
   
            and the trace of this matrix is
            $a^{2} b^{2} c^{2} + a^{2} b^{2} + a^{2} c^{2} + b^{2} c^{2} + a^{2} + b^{2} + c^{2} + \frac{1}{a^{2}} + \frac{1}{b^{2}} + \frac{1}{c^{2}} + \frac{1}{a^{2} b^{2}} + \frac{1}{a^{2} c^{2}} + \frac{1}{b^{2} c^{2}} + \frac{1}{a^{2} b^{2} c^{2}} + 4$
            which has 18 terms (including repeated terms).

	So, in particular, we know that there are only 20  $2 \times 2 \times 2$ boxed plane partitions: the minimal one, the maximal one, and 18 others. Then using the trace of our matrix, we have accounted for all 18 terms that correspond to the 18 perfect matchings which get squished by the squish map to a single loop.  These are the correct weights for plane partitions with 2-periodic weights.  
\end{example}

\section{Generating functions and specializations}
\label{sec:gf}

Note that the previous sections have used periodic edge weights and connections. Here we generalize to cover several natural weight functions for plane partition enumeration. 

\subsection{Arbitrary weights}

We first consider an \emph{arbitrary} nonzero weight function $w:G \rightarrow \mathbb{C}^*$. Note that every vertex $v$ of $G$ is a part of some propeller:  either $v$ is the center vertex of a propeller, or not (in which case, only one of the edges incident to $v$ is part of the propeller).

We modify the weight function $w$ by performing a so-called \emph{gauge transformation}: for each vertex $v$ which is not the center of a propeller; suppose that the edge $e$ connects $v$ to the center of the propeller.  Divide the weights of all of $v$'s incident edges by $w(e)$. This operation changes the dimer model partition function only by an overall constant, which we can ignore at least in the case where $G$ is a finite subgraph of the honeycomb graph.  Assume without loss of generality that this has been done.

Then, given two edges $e_1, e_2$ which get sent to $e$ under the squish map, let
$
w(e) = \sqrt{w(e_1) w(e_2)},
$
the geometric mean of the weights of $e_1$ and $e_2$.  Define a new weighting on $G$ by
$$
\widetilde{w}(e_1) = \frac{w(e_1)}{w(e)}, \qquad \widetilde{w}(e_2) = \frac{w(e_2)}{w(e)}.
$$
This weighting $\widetilde{w}$ has the property that $\widetilde{w}(e_1) = \widetilde{w}(e_2)^{-1}$, so we can push it through the squish map as before.

\subsection{Periodic weights on plane partitions}

One has to be careful with computing the weight of a perfect matching $m$ on any infinite graph, such as $H$. There are infinitely many edges in such a graph, so one would need to consider issues of convergence before multiplying all of the weights together to find the weight of a perfect matching. There is then a separate convergence issue when trying to sum over $m$.

However, our main interest is in plane partitions, which are in bijection with certain perfect matchings on $H$. Plane partitions come with a natural generating function, computed by MacMahon, as stated in Section \ref{sec:definitions}. We now define
$$M(a, q) = \prod_{i \geq 1} \left(\frac{1}{1-aq^{i}}\right)^i$$
for weighted plane partitions~\cite{macmahon}.  Indeed, there is a particular weighting which is uniquely relevant on the entire honeycomb lattice: the $\mathbb{Z}_2 \times \mathbb{Z}_2$ periodic weighting.

We assign the weights \emph{to the lattice points} as follows:
\begin{align}
w(i,j,k) = \begin{cases}
q & \text{ if } i-j \equiv 0, i-k \equiv 0 \pmod{2},\\
r & \text{ if } i-j \equiv 1, i-k \equiv 0 \pmod{2},\\
s & \text{ if } i-j \equiv 0, i-k \equiv 1 \pmod{2},\\
t & \text{ if } i-j \equiv 1, i-k \equiv 1 \pmod{2}.
\end{cases}
\label{box labels}
\end{align}

There are a variety of edge weight functions on the dimer model which correspond to this; one which is well behaved with respect to the squish map is (partially) shown in Figure~\ref{preconnection}. We previously defined the weight function on a finite portion of the honeycomb graph (as in the monochromatic weighting on Figure \ref{single}). To extend this weighting to the entire plane, set $p^3=q$ and include constants $k_1$, $k_2$, and $k_3$. 

Choose each of $k_1$, $k_2$, and $k_3$ so that the desired region of the plane has the appropriate edge weights. (Generally we choose this in such a way that the edge weights near the center of the region in question are all of low degree.) This is comparable to choosing the weight ``1" on a particular place in a given column of horizontal edges from Example \ref{monochromatic}. The result will be that each `strip' of hexagons (in each vertical, northeast/southwest, or northwest/southeast diagonal, as in figure \ref{preconnection}) will have a weight of $1$ in the desired location, with increasing and decreasing powers traveling either direction away from the center. 

We can decompose this weighting into two pieces, as shown on the bottom of Figure~\ref{preconnection}.  The first piece becomes the $\SL_2$ connection, as described above and seen in the lower right of the figure. The second piece becomes the scalar weights, which is shown in the lower left. For further details on various weight functions of the honeycomb graph, see~\cite{young:thesis}. 

    \begin{theorem}
            Using the weight function given in Figure \ref{single} and a monodromy given by placing $\alpha$ on every northeast/southwest edge, $\beta$ on every northwest/southeast edge, and $\gamma$ on every horizontal (east/west) edge, then the squish map is measure preserving.
            \label{bigtheorem}
    \end{theorem} 

\begin{corollary}
	Using $\alpha$, $\beta$, and $\gamma$ as described in Theorem \ref{bigtheorem}, we can find the generating function for a given $m \times n \times o$ boxed double dimer plane partition.
\end{corollary}

\begin{remark}
	As seen in \cite{kenyon:conformal}, we can compute the probability of a given edge being present in the single dimer model by taking the determinant of a certain Kasteleyn matrix. When we use this procedure in conjunction with the squish map, we can now determine the probability that a given edge will appear as a doubled edge in the double dimer model. If said doubled edge is present, then its preimage under the squish map would have two edges mapping to the doubled edge, so the overall probability would be the product of the two individual probabilities that each edge would be present in the single dimer model. 
\end{remark}

Define $Q := qrst$, and $\widetilde{M}(x,y):=M(x,y)M(x^{-1},y)$. The generating function for plane partitions in which $q,r,s,t$ mark the boxes in as in equation \ref{box labels} was computed in \cite{young:thesis} (where it was denoted $Z_{\Z_2 \times \Z_2}$); it is

\begin{align}
Z_Q = M(1,Q)^4
\frac{
\widetilde{M}(rs,Q)
\widetilde{M}(st,Q)
\widetilde{M}(tr,Q)
}{
\widetilde{M}(-r,Q)
\widetilde{M}(-s,Q)
\widetilde{M}(-t,Q)
\widetilde{M}(-rst,Q)
}.
\label{benGF}
\end{align}
Under the specialization $r=s=t=q$ (and hence $Q=q^4$) we recover MacMahon's generating function for plane partitions by a slightly delicate manipulation of formal power series.  In \cite{young:squish} it is observed that under the specialization $r=s=t=-1$, the above generating function specializes to $M(1,-q)$, and this latter statement is proven using the squish map (without the $\SL_2$ double dimer model).

\begin{corollary}[to Theorem \ref{bigtheorem}]
	The partition function for the $\SL_2$ double dimer model is $Z_Q$ from equation \ref{benGF}.
\end{corollary}

\begin{figure}[ht]	
	\begin{tikzpicture}[scale = 0.925]
		\node at (0,5) {
			\scriptsize
			$$\begin{tikzpicture}[scale = 1.25*0.95]
				\def\hexagon{+(0:1cm) -- +(60:1cm) -- +(120:1cm) -- +(180:1cm) -- +(240:1cm) -- +(300:1cm) -- cycle}
				\def\hexagonweights{+(0:1cm) -- node [anchor = north east]{$sk_2$} +(60:1cm) -- node [anchor = north]{$tpk_3$} +(120:1cm) -- node [anchor = north west]{$rk_1$} +(180:1cm) -- node [anchor = south west]{$spk_2$} +(240:1cm) -- node [anchor = south]{$tk_3$} +(300:1cm) -- node [anchor = south east]{$rpk_1$}cycle;}
				\def \squish {+(0: 1cm) -- 
					++(60: 1cm) -- +(60: .155cm) -- +(240: 0cm) -- 
					++(180:1cm) -- +(120: .155cm) -- +(240: 0cm) -- 
					++(240: 1cm) -- +(180: .155cm) -- + (240: 0cm) -- 
					++(300: 1cm) -- +(240: .155cm) -- +(240: 0cm) -- 
					++(360: 1cm) -- +(300: .155cm) -- +(240: 0cm) -- 
					++(420: 1cm) -- +(360: .155cm) -- +(240: 0cm)}
				
				\clip[draw] (1.73,1) circle (2.6cm);
				
				\draw (1.73, 1) \squish;
				
				\draw (3.46, 2) \squish; 
				\def\hexagonweights{+(0:1cm) -- node [anchor = north east]{$p^{-1}k_2$} +(60:1cm) -- node [anchor = north]{$t^2pk_3$} +(120:1cm) -- node [anchor = north west]{$rk_1$} +(180:1cm) -- node [anchor = south west]{$k_2$} +(240:1cm) -- node [anchor = south]{$tk_3$} +(300:1cm) -- node [anchor = south east]{$rpk_1$}cycle;}
				
				\draw (1.73, 3) \squish; 
				\def\hexagonweights{+(0:1cm) -- node [anchor = north east]{$p^{-1}k_2$} +(60:1cm) -- node [anchor = north]{$t^2p^2k_3$} +(120:1cm) -- node [anchor = north west]{$p^{-1}k_1$} +(180:1cm) -- node [anchor = south west]{$k_2$} +(240:1cm) -- node [anchor = south]{$t^2pk_3$} +(300:1cm) -- node [anchor = south east]{$k_1$}cycle;}
				
				\draw (1.73, -1) \squish; 
				\def\hexagonweights{+(0:1cm) -- node [anchor = north east]{$sk_2$} +(60:1cm) -- node [anchor = north]{$k_3$} +(120:1cm) -- node [anchor = north west]{$rk_1$} +(180:1cm) -- node [anchor = south west]{$spk_2$} +(240:1cm) -- node [anchor = south]{$p^{-1}k_3$} +(300:1cm) -- node [anchor = south east]{$rpk_1$}cycle;}
				
				\draw (0,0) \squish; 
				\def\hexagonweights{+(0:1cm) -- node [anchor = north east]{$s^2pk_2$} +(60:1cm) -- node [anchor = north]{$k_3$} +(120:1cm) -- node [anchor = north west]{$p^{-1}k_1$} +(180:1cm) -- node [anchor = south west]{$s^2p^2k_2$} +(240:1cm) -- node [anchor = south]{$p^{-1}k_3$} +(300:1cm) -- node [anchor = south east]{$k_1$}cycle;}

				\draw (0, 2) \squish; 
				\def\hexagonweights{+(0:1cm) -- node [anchor = north east]{$sk_2$} +(60:1cm) -- node [anchor = north]{$tpk_3$} +(120:1cm) -- node [anchor = north west]{$p^{-1}k_1$} +(180:1cm) -- node [anchor = south west]{$spk_2$} +(240:1cm) -- node [anchor = south]{$tk_3$} +(300:1cm) -- node [anchor = south east]{$k_1$}cycle;}
				
				\draw (-1.73, 1) \squish;
				\draw(3.46, 0) \squish; 
				\def\hexagonweights{+(0:1cm) -- node [anchor = north east]{$p^{-1}k_2$} +(60:1cm) -- node [anchor = north]{$k_3$} +(120:1cm) -- node [anchor = north west]{$r^2pk_1$} +(180:1cm) -- node [anchor = south west]{$k_2$} +(240:1cm) -- node [anchor = south]{$p^{-1}k_3$} +(300:1cm) -- node [anchor = south east]{$r^2p^2k_1$}cycle;}
				\draw(0, -2) \squish;
				\draw(0,4) \squish;
				\draw(-1.73, 3) \squish;
				\draw(-1.73, -1) \squish;
				\draw(1.73, 5) \squish;
				\draw(1.73, -3) \squish;
				\draw (3.46, 4) \squish;
				\draw (3.46, -2) \squish;
				\draw (3.46 + 1.73, 1) \squish;
				\draw (3.46 + 1.73, 3) \squish;
				\draw (3.46 + 1.73, -1) \squish;
				
				\normalsize
				\node at (1.73, 2) {$t$};
				\node at (1.73, 0) {$t$};
				\node at (0, 1) {$t$};
				\node at (1.73*2, 1) {$t$};
				\node at (1.73 * 0.5, 0.5) {$s$};
				\node at (1.73 * 0.5, 1.5) {$r$};
				\node at (1.73 * 0.5, 2.5) {$s$};
				\node at (1.73 * 0.5, -0.5) {$r$};
				\node at (1.73 * 0.5 + 1.73, 0.5) {$r$};
				\node at (1.73 * 0.5 + 1.73, 2.5) {$r$};
				\node at (1.73 * 0.5 + 1.73, 1.5) {$s$};
				\node at (1.73 * 0.5 + 1.73, -0.5) {$s$};
				\large
				
				\node[blue] at (1.73, 1) {$q$};
				\node[blue] at (1.73, -1) {$q$};
				\node[blue] at (1.73, 3) {$q$};
				\node[blue] at (0, 0) {$q$};
				\node[blue] at (0, 2) {$q$};
				\node[blue] at (1.73*2, 0) {$q$};
				\node[blue] at (1.73*2, 2) {$q$};

			\end{tikzpicture}$$
			
			\normalsize 
			
		};
		\node at (-5, 0) {
			\small
			$$\begin{tikzpicture}[scale = 1.5*0.95]
				\def\hexagon{+(0:1cm) -- +(60:1cm) -- +(120:1cm) -- +(180:1cm) -- +(240:1cm) -- +(300:1cm) -- cycle}
				\def\hexagonweights{+(0:1cm) -- node [anchor = east]{\scriptsize$s^{1/2}$} +(60:1cm) -- node [anchor = north]{\scriptsize$t^{-1/2}$} +(120:1cm) -- node [anchor = west]{\scriptsize$r^{1/2}$} +(180:1cm) -- node [anchor = west]{\scriptsize$s^{-1/2}$} +(240:1cm) -- node [anchor = south]{\scriptsize$t^{1/2}$} +(300:1cm) -- node [anchor = east]{\scriptsize$r^{-1/2}$}cycle;}
				\def \squish {+(0: 1cm) -- 
					++(60: 1cm) -- +(60: .155cm) -- +(240: 0cm) -- 
					++(180:1cm) -- +(120: .155cm) -- +(240: 0cm) -- 
					++(240: 1cm) -- +(180: .155cm) -- + (240: 0cm) -- 
					++(300: 1cm) -- +(240: .155cm) -- +(240: 0cm) -- 
					++(360: 1cm) -- +(300: .155cm) -- +(240: 0cm) -- 
					++(420: 1cm) -- +(360: .155cm) -- +(240: 0cm)}
				
				\clip[draw] (1.73,1) circle (2.2cm);
				
				\draw (1.73, 1) \squish;
				\draw (1.73, 1) \hexagonweights;
				\draw (3.46, 2) \squish;
				\draw (3.46, 2) \hexagonweights;
				\draw (1.73, 3) \squish;
				\draw (1.73, 3) \hexagonweights;
				\draw (1.73, -1) \squish;
				\draw (1.73, -1) \hexagonweights;
				\draw (0,0) \squish;
				\draw (0,0) \hexagonweights;
				\draw (1.73, 3) \squish;
				\draw (1.73, 3) \hexagonweights;
				\draw (0, 2) \squish;
				\draw (0, 2) \hexagonweights;
				\draw (-1.73, 1) \squish;
				\draw(3.46, 0) \squish;
				\draw(3.46, 0) \hexagonweights;
				\draw(0, -2) \squish;
				\draw(0,4) \squish;
				\draw(0,4) \hexagonweights;
				\draw(-1.73, 3) \squish;
				\draw(-1.73, -1) \squish;
				\draw(1.73, 5) \squish;
				\draw(1.73, -3) \squish;
				\draw (3.46, 4) \squish;
				\draw (3.46, -2) \squish;
				\draw (3.46 + 1.73, 1) \squish;
				\draw (3.46 + 1.73, 3) \squish;
				\draw (3.46 + 1.73, -1) \squish;

			\end{tikzpicture}$$
			
			\normalsize 
			
		};
		\node at (5, 0) {
			\scriptsize
			$$\begin{tikzpicture}[scale = 1.5*0.95]
				\def\hexagon{+(0:1.155cm) -- +(60:1.155cm) -- +(120:1.155cm) -- +(180:1.155cm) -- +(240:1.155cm) -- +(300:1.155cm) -- cycle}
				\def\hexagonweights{+(0:1.155cm) -- node [anchor = north east]{$s^{1/2}k_2$} +(60:1.155cm) -- node [anchor = north]{$t^{3/2}k_3 p$} +(120:1.155cm) -- node [anchor = north west]{$r^{1/2}k_1$} +(180:1.155cm) -- node [anchor = south west]{$s^{3/2}k_2 p$} +(240:1.155cm) -- node [anchor = south]{$t^{1/2}k_3$} +(300:1.155cm) -- node [anchor = south east]{$r^{3/2}k_1 p$}cycle;}

				\clip[draw] (1.73,1) circle (2.2cm);

				\def\hexagonweights{+(0:1.155cm) -- node [anchor = east]{$s^{1/2}k_2$} +(60:1.155cm) -- node [anchor = north]{} +(120:1.155cm) -- node [anchor = north west]{$r^{1/2}k_1$} +(180:1.155cm) -- node [anchor = south west]{} +(240:1.155cm) -- node [anchor = south]{$t^{1/2}k_3$} +(300:1.155cm) -- node [anchor = south east]{}cycle;}
				\draw (1.73, 1) \hexagonweights;
				
				\def\hexagonweights{+(0:1.155cm) -- node [anchor = north east]{} +(60:1.155cm) -- node [anchor = north]{$t^{3/2}k_3 p$} +(120:1.155cm) -- node [anchor = north west]{$r^{3/2}k_1p$} +(180:1.155cm) -- node [anchor = south west]{} +(240:1.155cm) -- node [anchor = south]{$t^{1/2}k_3$} +(300:1.155cm) -- node [anchor = south east]{}cycle;}
				\draw (3.46, 2) \hexagonweights;
				
				\def\hexagonweights{+(0:1.155cm) -- node [anchor = north east]{} +(60:1.155cm) -- node [anchor = north]{} +(120:1.155cm) -- node [anchor = north west]{$r^{3/2}k_1 p$} +(180:1.155cm) -- node [anchor = south west]{} +(240:1.155cm) -- node [anchor = south]{$t^{1/2}k_3$} +(300:1.155cm) -- node [anchor = south east]{}cycle;}
				\draw(3.46, 0) \hexagonweights;
				
				\def\hexagonweights{+(0:1.155cm) -- node [anchor = north east]{$s^{1/2}k_2$} +(60:1.155cm) -- node [anchor = north]{} +(120:1.155cm) -- node [anchor = north west]{$r^{1/2}k_1$} +(180:1.155cm) -- node [anchor = south west]{} +(240:1.155cm) -- node [anchor = south]{$t^{3/2}k_3 p$} +(300:1.155cm) -- node [anchor = south east]{}cycle;}
				\draw (1.73, 3) \hexagonweights;
				
				\def\hexagonweights{+(0:1.155cm) -- node [anchor = north east]{$s^{3/2}k_2p$} +(60:1.155cm) -- node [anchor = north]{} +(120:1.155cm) -- node [anchor = north west]{$r^{1/2}k_1$} +(180:1.155cm) -- node [anchor = south west]{$s^{3/2}k_2 p$} +(240:1.155cm) -- node [anchor = south]{} +(300:1.155cm) -- node [anchor = south east]{$r^{3/2}k_1 p$}cycle;}
				\draw (1.73, -1) \hexagonweights;

				\def\hexagonweights{+(0:1.155cm) -- node [anchor = north east]{$s^{3/2}k_2 p$} +(60:1.155cm) -- node [anchor = north]{} +(120:1.155cm) -- node [anchor = north west]{} +(180:1.155cm) -- node [anchor = south west]{$s^{3/2}k_2 p$} +(240:1.155cm) -- node [anchor = south]{$t^{1/2}k_3$} +(300:1.155cm) -- node [anchor = south east]{}cycle;}
				\draw (0,0) \hexagonweights;

				\def\hexagonweights{+(0:1.155cm) -- node [anchor = north east]{$s^{1/2}k_2$} +(60:1.155cm) -- node [anchor = north]{$t^{3/2}k_3 p$} +(120:1.155cm) -- node [anchor = north west]{} +(180:1.155cm) -- node [anchor = south west]{} +(240:1.155cm) -- node [anchor = south]{$t^{3/2}k_3p$} +(300:1.155cm) -- node [anchor = south east]{}cycle;}
				\draw (0, 2) \hexagonweights;

				\node[blue] at (1.73, 1) {\LARGE$Q^2$};
				
			\end{tikzpicture}$$
			
			\normalsize 
			
		};
	\end{tikzpicture}

	\caption{Top: 2-periodic weights for the single dimer model. Bottom: $\SL_2$ connection and scalar weight for the corresponding double dimer model.}
	\label{preconnection}
	\label{sqrtweights}
\end{figure}
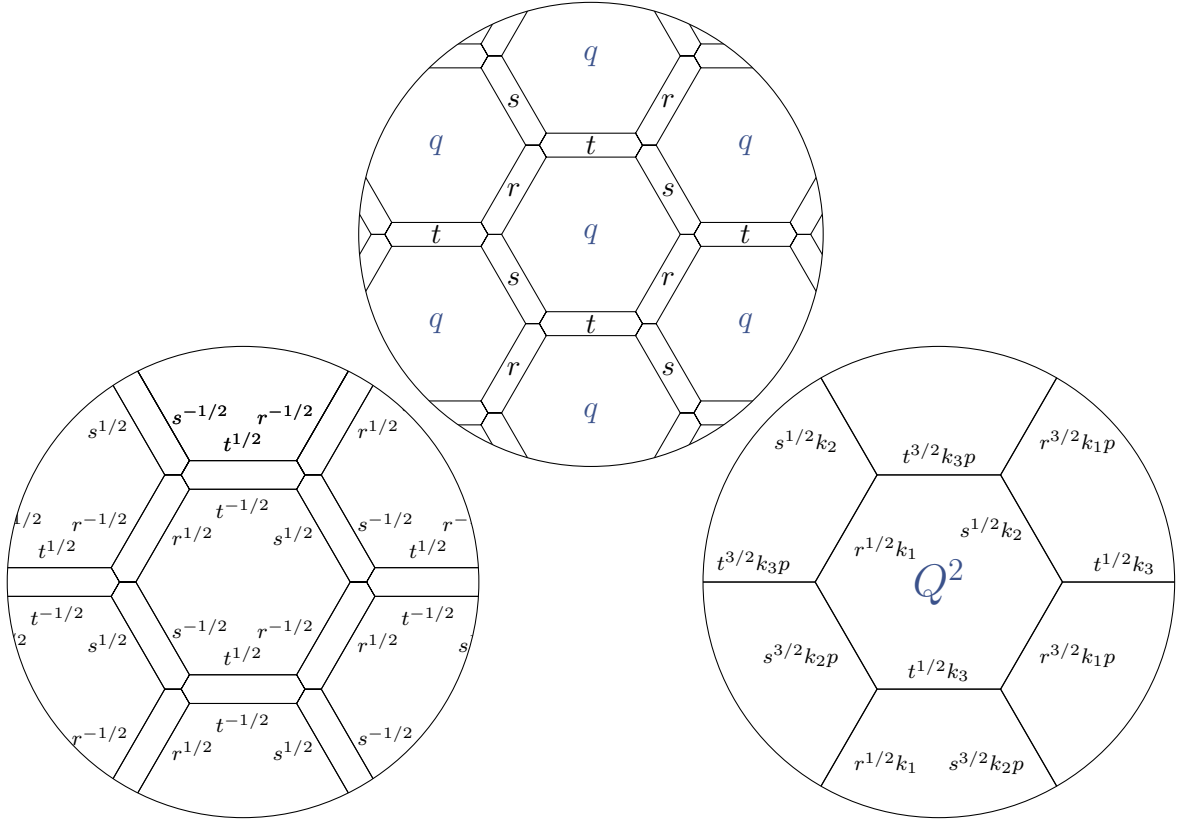

\section{Examples}
\label{rstsection}
We now work out two examples in which we squish boxed $2\times2\times2$ and $2 \times 2 \times 4$ plane partitions, with boxes weighted by $w(i,j,k)$.  These do not have nice closed-form generating functions akin to MacMahon's generating function, but we can nonetheless evaluate them with our techniques.  We also investigate the specializations of our formula at primitive roots of unity, recovering and extending the results of~\cite{young:thesis}, as well as a new conjecture.

\begin{figure}[ht]
	
	\small
	$$\begin{tikzpicture}[scale = 1.0]
		\def\hexagon{+(0:1cm) -- +(60:1cm) -- +(120:1cm) -- +(180:1cm) -- +(240:1cm) -- +(300:1cm) -- cycle}
		\def\hexagonweights{+(0:1cm) -- node [anchor = east]{$c$} +(60:1cm) -- node [anchor = north]{$a^{-1}$} +(120:1cm) -- node [anchor = west]{$b^{-1}$} +(180:1cm) -- node [anchor = west]{$c^{-1}$} +(240:1cm) -- node [anchor = south]{$a$} +(300:1cm) -- node [anchor = east]{$b$}cycle;}
		\def \squish {+(0: 1cm) -- 
			++(60: 1cm) -- +(60: .155cm) -- +(240: 0cm) -- 
			++(180: 1cm) -- +(120: .155cm) -- +(240: 0cm) -- 
			++(240: 1cm) -- +(180: .155cm) -- + (240: 0cm) -- 
			++(300: 1cm) -- +(240: .155cm) -- +(240: 0cm) -- 
			++(360: 1cm) -- +(300: .155cm) -- +(240: 0cm) -- 
			++(420: 1cm) -- +(360: .155cm) -- +(240: 0cm)}

		\draw [line width = 4pt, green] (0, 2*1.73)  (2, 2*1.73) -- (3, 3*1.73)  (2, 4*1.73) -- (0, 4*1.73)  (-1, 3*1.73) -- (0, 2*1.73); 
		
		\draw [line width = 4pt, green] (-1, 3*1.73) ++(180: 0.25cm)  +(240: 0.25cm) +(240: 0cm) -- ++(120: 0.25cm)  ++(60: 2cm);
		\draw [line width = 4pt, green] (0, 2*1.73)  ++(240: 0.25cm) -- +(300: 0.25cm) +(240: 0cm)  ++(180: 0.25cm) -- ++(120: 2cm);
		\draw [line width = 4pt, green] (2, 2*1.73)  ++(300: 0.25cm)  +(0: 0.25cm)   +(240: 0cm) -- ++(240: 0.25cm)  ++(180: 2cm);
		\draw [line width = 4pt, green] (3, 3*1.73)  ++(0: 0.25cm)   -- +(60: 0.25cm)  +(240: 0cm)  ++(300: 0.25cm) -- ++(240: 2cm);
		\draw [line width = 4pt, green] (2, 4*1.73)  ++(60: 0.25cm)   +(120: 0.25cm) +(240: 0cm) -- ++(0: 0.25cm)  ++(300: 2cm);
		\draw [line width = 4pt, green] (0, 4*1.73)  ++(120: 0.25cm) -- +(180: 0.25cm) +(240: 0cm)  ++(60: 0.25cm) -- ++(0: 2cm);

		\draw (0, 2*1.73) -- (2, 2*1.73) -- (3, 3*1.73) -- (2, 4*1.73) -- (0, 4*1.73) -- (-1, 3*1.73) -- cycle; 
		
		\draw (-1, 3*1.73) -- ++(180: 0.25cm) -- +(240: 0.25cm) -- +(240: 0cm) -- ++(120: 0.25cm) -- ++(60: 2cm);
		\draw (0, 2*1.73) -- ++(240: 0.25cm) -- +(300: 0.25cm) -- +(240: 0cm) -- ++(180: 0.25cm) -- ++(120: 2cm);
		\draw (2, 2*1.73) -- ++(300: 0.25cm) -- +(0: 0.25cm) -- +(240: 0cm) -- ++(240: 0.25cm) -- ++(180: 2cm);
		\draw (3, 3*1.73) -- ++(0: 0.25cm) -- +(60: 0.25cm) -- +(240: 0cm) -- ++(300: 0.25cm) -- ++(240: 2cm);
		\draw (2, 4*1.73) -- ++(60: 0.25cm) -- +(120: 0.25cm) -- +(240: 0cm) -- ++(0: 0.25cm) -- ++(300: 2cm);
		\draw (0, 4*1.73) -- ++(120: 0.25cm) -- +(180: 0.25cm) --  +(240: 0cm) -- ++(60: 0.25cm) -- ++(0: 2cm);

	\end{tikzpicture}
	\qquad 
	\begin{tikzpicture}[scale = 1.0]
		\def\hexagon{+(0:1cm) -- +(60:1cm) -- +(120:1cm) -- +(180:1cm) -- +(240:1cm) -- +(300:1cm) -- cycle}
		\def\hexagonweights{+(0:1cm) -- node [anchor = east]{$c$} +(60:1cm) -- node [anchor = north]{$a^{-1}$} +(120:1cm) -- node [anchor = west]{$b^{-1}$} +(180:1cm) -- node [anchor = west]{$c^{-1}$} +(240:1cm) -- node [anchor = south]{$a$} +(300:1cm) -- node [anchor = east]{$b$}cycle;}
		\def \squish {+(0: 1cm) -- 
			++(60: 1cm) -- +(60: .155cm) -- +(240: 0cm) -- 
			++(180: 1cm) -- +(120: .155cm) -- +(240: 0cm) -- 
			++(240: 1cm) -- +(180: .155cm) -- + (240: 0cm) -- 
			++(300: 1cm) -- +(240: .155cm) -- +(240: 0cm) -- 
			++(360: 1cm) -- +(300: .155cm) -- +(240: 0cm) -- 
			++(420: 1cm) -- +(360: .155cm) -- +(240: 0cm)}

		\draw [line width = 4pt, green] (0, 2*1.73) -- (2, 2*1.73)  (3, 3*1.73) -- (2, 4*1.73)  (0, 4*1.73) -- (-1, 3*1.73)  (0, 2*1.73); 
		
		\draw [line width = 4pt, green] (-1, 3*1.73) ++(180: 0.25cm) -- +(240: 0.25cm) +(240: 0cm)  ++(120: 0.25cm) -- ++(60: 2cm);
		\draw [line width = 4pt, green] (0, 2*1.73)  ++(240: 0.25cm)  +(300: 0.25cm) +(240: 0cm) -- ++(180: 0.25cm)  ++(120: 2cm);
		\draw [line width = 4pt, green] (2, 2*1.73)  ++(300: 0.25cm) -- +(0: 0.25cm)   +(240: 0cm)  ++(240: 0.25cm) -- ++(180: 2cm);
		\draw [line width = 4pt, green] (3, 3*1.73)  ++(0: 0.25cm)    +(60: 0.25cm)  +(240: 0cm) -- ++(300: 0.25cm)  ++(240: 2cm);
		\draw [line width = 4pt, green] (2, 4*1.73)  ++(60: 0.25cm)  -- +(120: 0.25cm) +(240: 0cm)  ++(0: 0.25cm)  -- ++(300: 2cm);
		\draw [line width = 4pt, green] (0, 4*1.73)  ++(120: 0.25cm)  +(180: 0.25cm) +(240: 0cm) -- ++(60: 0.25cm)  ++(0: 2cm);

		\draw (0, 2*1.73) -- (2, 2*1.73) -- (3, 3*1.73) -- (2, 4*1.73) -- (0, 4*1.73) -- (-1, 3*1.73) -- cycle; 
		
		\draw (-1, 3*1.73) -- ++(180: 0.25cm) -- +(240: 0.25cm) -- +(240: 0cm) -- ++(120: 0.25cm) -- ++(60: 2cm);
		\draw (0, 2*1.73) -- ++(240: 0.25cm) -- +(300: 0.25cm) -- +(240: 0cm) -- ++(180: 0.25cm) -- ++(120: 2cm);
		\draw (2, 2*1.73) -- ++(300: 0.25cm) -- +(0: 0.25cm) -- +(240: 0cm) -- ++(240: 0.25cm) -- ++(180: 2cm);
		\draw (3, 3*1.73) -- ++(0: 0.25cm) -- +(60: 0.25cm) -- +(240: 0cm) -- ++(300: 0.25cm) -- ++(240: 2cm);
		\draw (2, 4*1.73) -- ++(60: 0.25cm) -- +(120: 0.25cm) -- +(240: 0cm) -- ++(0: 0.25cm) -- ++(300: 2cm);
		\draw (0, 4*1.73) -- ++(120: 0.25cm) -- +(180: 0.25cm) --  +(240: 0cm) -- ++(60: 0.25cm) -- ++(0: 2cm);

	\end{tikzpicture}$$
	
	\normalsize 
	
	\caption{The only two configurations of a $2 \times 2 \times 2$ single dimer model that do not squish to a single hexagon loop}
	\label{squishsingle}
\end{figure}
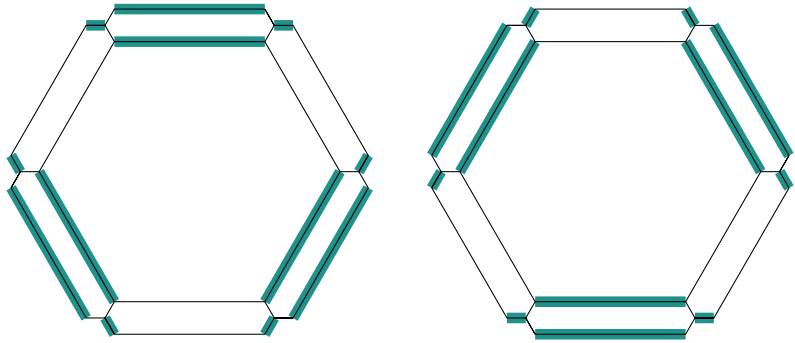

\begin{example}[$2 \times 2 \times 2$]
	Consider the path around a single hexagon that is the result of a particular single dimer configuration sent through the squish map, or a $2 \times 2 \times 2$ boxed plane partition in the double dimer model. The two configurations on $H_{2, 2, 2}$ that do not give rise to a loop under the squish map are pictured in Figure \ref{squishsingle}. The right graphic has a total weight of 1, and the right graphic has a total weight of $Q^2$. Then the remaining 18 partitions are represented by the polynomial $1 + q + qr + qs + qt + qrs + qrt + qst + 4qrst + qr^2st + qrs^2t + qrst^2 + q r^{2} s^{2} t + q r^{2} s t^{2} + q r s^{2} t^{2} + q r^{2} s^{2} t^{2} + q^2r^2s^2t^2$.

	We want to draw a comparison between $Z_Q$ and the matrix model under the squish map, so consider the trace of the product of terms corresponding to the single hexagon loop, $\displaystyle\frac{\eta}{a^{2} b^{2} c^{2}}$, where $\eta := a^{4} b^{4} c^{4} + a^{4} b^{4} c^{2} + a^{4} b^{2} c^{4} + a^{2} b^{4} c^{4} + a^{4} b^{2} c^{2} + a^{2} b^{4} c^{2} + a^{2} b^{2} c^{4} + 4 a^{2} b^{2} c^{2} + a^{2} b^{2} + a^{2} c^{2} + b^{2} c^{2} + a^{2} + b^{2} + c^{2} + 1$. 
	\normalsize Note that this is only the connection in the $\SL_2(\C)$ double dimer model, not also the weights. So let the weight on every horizontal edge be $$Q^{(\text{height of that edge})} = (qa^2b^2c^2)^{(\text{height of that edge})},$$ as in Figure \ref{single}. Then the total contribution (the connection and the weight) of that single hexagon is $$qa^2b^2c^2 \left(\frac{\eta}{a^{2} b^{2} c^{2}}\right)= q\cdot \eta.$$
	
	Now consider a mapping from $Z_Q$ to the matrices, where $a^2 = r, b^2 = s$, and $c^2 = t$. Then $Q = qrst = qa^2b^2c^2$. Then making the replacements in the statement above, we get \small$$q\left(r^2s^2t^2 + r^2s^2t + r^2st^2 + rs^2t^2 + r^2st + rs^2t + rst^2 + 4rst + rs + st + rt + r + s + t + 1\right).$$\normalsize
	
	Then this matches what we got for $Z_Q$ when we include the $1$ and $Q^2$ terms from the two non-loop configurations.  
	\label{singlehexex}
\end{example}
\begin{example}[$2 \times 2 \times 4$]
	
	For another example, we consider the case with two hexagons, or a $2 \times 2 \times 4$ boxed plane partition. Our generating function is 
	\small
	\begin{align*}\begin{split}
			GF_{2 \times 2 \times 4} = &\ \ q^{4} r^{4} s^{4} t^{4} + q^{3} r^{4} s^{4} t^{4} + q^{3} r^{4} s^{4} t^{3} + q^{3} r^{4} s^{3} t^{4} + q^{3} r^{3} s^{4} t^{4} + q^{3} r^{4} s^{3} t^{3} + q^{3} r^{3} s^{4} t^{3}  \\
			&+ q^{2} r^{4} s^{4} t^{3} + q^{3} r^{3} s^{3} t^{4} + q^{2} r^{4} s^{4} t^{2} + 4  q^{3} r^{3} s^{3} t^{3} + q^{2} r^{4} s^{3} t^{3} + q^{2} r^{3} s^{4} t^{3} + q^{3} r^{3} s^{3} t^{2}  \\
			&+ q^{2} r^{4} s^{3} t^{2} + q^{2} r^{3} s^{4} t^{2} + q^{3} r^{3} s^{2} t^{3} + q^{3} r^{2} s^{3} t^{3} + 3  q^{2} r^{3} s^{3} t^{3} + q^{3} r^{3} s^{2} t^{2} + q^{3} r^{2} s^{3} t^{2} \\
			& + 4  q^{2} r^{3} s^{3} t^{2} + q^{3} r^{2} s^{2} t^{3} + 2  q^{2} r^{3} s^{2} t^{3} + 2  q^{2} r^{2} s^{3} t^{3} + q^{2} r^{3} s^{3} t + q^{3} r^{2} s^{2} t^{2} + 3  q^{2} r^{3} s^{2} t^{2} \\
			& + 3  q^{2} r^{2} s^{3} t^{2} + 3  q^{2} r^{2} s^{2} t^{3} + q^{2} r^{3} s^{2} t + q^{2} r^{2} s^{3} t + 9  q^{2} r^{2} s^{2} t^{2} + q^{2} r^{2} s t^{3} + q^{2} r s^{2} t^{3}	\\
			& + 3  q^{2} r^{2} s^{2} t + 3  q^{2} r^{2} s t^{2} + 3  q^{2} r s^{2} t^{2} + q r^{2} s^{2} t^{2} + q^{2} r s t^{3} + 2  q^{2} r^{2} s t + 2  q^{2} r s^{2} t + q r^{2} s^{2} t\\
			& + 4  q^{2} r s t^{2} + q r^{2} s t^{2} + q r s^{2} t^{2} + 3  q^{2} r s t + q r^{2} s t + q r s^{2} t + q^{2} r t^{2} + q^{2} s t^{2} + q r s t^{2}\\
			& + q^{2} r t + q^{2} s t + 4  q r s t + q^{2} t^{2} + q r s + q^{2} t + q r t + q s t + q r + q s + q t + q + 1.
		\end{split}\end{align*}
	\normalsize
	
	This generating function, however, contains all 105 of the terms given by a $2 \times 2 \times 4$ four-colored boxed plane partition. To compare with the matrix version, we need to be careful which terms squish to which configurations. See Figure \ref{twohex}. 
	
	\begin{figure}[ht]
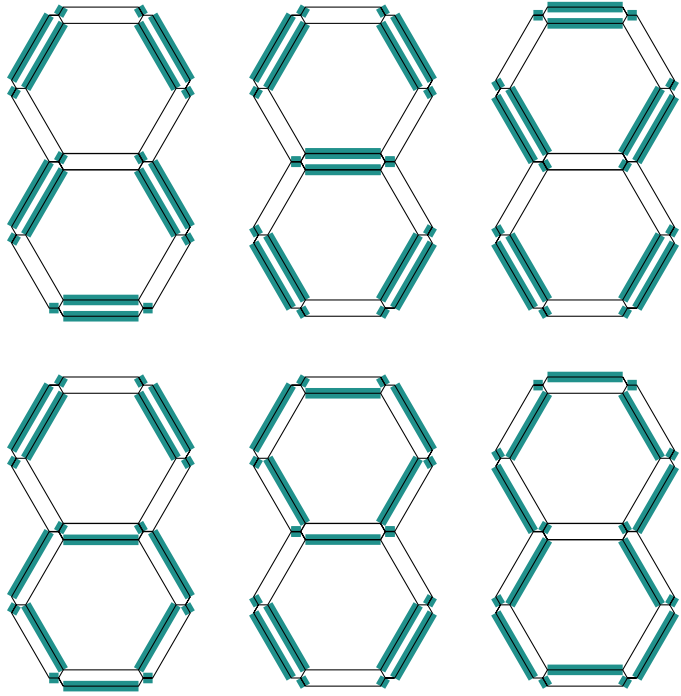

		\small
		$$
$$
		
		\normalsize 
		\caption{A representation of of each possible result of pushing the $2 \times 2 \times 4$ single dimer configuration through the squish map}
		\label{twohex}
	\end{figure}
	
	Note that the top row of configurations in Figure \ref{twohex} has no loops (only doubled edges), so there is no contribution from the $\SL_2(\C)$ connection. So we get a contribution only from the weight, and we get $1$ for the configuration on the left, $Q^2$ for the configuration in the middle, and $Q^4$ for the configuration on the right. Since there is only one configuration in the single dimer model that squishes to each of these they must each account for only one term each in the generating function $GF_{2 \times 2 \times 4}$, namely $1$, $q^2r^2s^2t^2$, and $q^4r^4s^4t^4$. 
	
	Now we examine the second line of Figure \ref{twohex}. From the $2 \times 2\times 2$ (Example \ref{singlehexex}) case, we see that we get 18 terms in the generating function for a single loop. So each of the first two figures in the second line correspond to 18 terms in the generating function, but to determine which terms we need to include the $Q$-weights. So the lower left figure has two horizontal edges, the lower getting a weight of $1$ and the upper getting a weight of $Q$. So this diagram corresponds with the terms $Q\left(\right.$the single hexagon generating function seen previously$) = Q \left(\right.r^2s^2t^2 + r^2s^2t + r^2st^2 + rs^2t^2 + r^2st + rs^2t + rst^2 + 4rst + rs + st + rt + r + s + t + 1\left.\right)$. Then the lower-middle diagram gets a weight of $Q^2 \cdot Q = Q^3$, so its generating function is $Q^3(r^2s^2t^2 + r^2s^2t + r^2st^2 + rs^2t^2 + r^2st + rs^2t + rst^2 + 4rst + rs + st + rt + r + s + t + 1)$. Then we have the $105$ monomials from $GF_{2 \times 2 \times 4}$, and subtract off the three monomials corresponding to the no-loop configurations, and the $2\cdot 18$ monomials just described corresponding to single hexagon loop configurations, and we are left with the following $66$ terms: $(r^{4} s^{4} t^{2} + r^{4} s^{4} t + r^{4} s^{3} t^{2} + r^{3} s^{4} t^{2} + r^{4} s^{3} t + r^{3} s^{4} t + 3 r^{3} s^{3} t^{2} + 4 r^{3} s^{3} t + 2 r^{3} s^{2} t^{2} + 2 r^{2} s^{3} t^{2} + r^{3} s^{3} + 3 r^{3} s^{2} t + 3 r^{2} s^{3} t + 3 r^{2} s^{2} t^{2} + r^{3} s^{2} + r^{2} s^{3} + 8 r^{2} s^{2} t + r^{2} s t^{2} + r s^{2} t^{2} + 3 r^{2} s^{2} + 3 r^{2} s t + 3 r s^{2} t + r s t^{2} + 2 r^{2} s + 2 r s^{2} + 4 r s t + 3 r s + r t + s t + r + s + t + 1) q^{2} t $, which corresponds to a loop around two hexagons, as in the lower right of Figure \ref{twohex}. 
	
	To compare with the corresponding contribution of the $\SL_2$ double dimer model, we first take the matrix product of a monodromy around both hexagons, then take the trace and multiply by the weight $(qa^2b^2c^2)^2$.
	\begin{align*}\begin{split}
		(qa^2b^2c^2)^2 \cdot\left(\frac{1}{a^{4} b^{4} c^{2}}\cdot\left[a^{8} b^{8} c^{4} + a^{8} b^{8} c^{2} + a^{8} b^{6} c^{4} + a^{6} b^{8} c^{4} + a^{8} b^{6} c^{2} + a^{6} b^{8} c^{2} \right.\right.\\
		+ 3 a^{6} b^{6} c^{4} + 4 a^{6} b^{6} c^{2} + 2 a^{6} b^{4} c^{4} + 2 a^{4} b^{6} c^{4} + a^{6} b^{6} + 3 a^{6} b^{4} c^{2} + 3 a^{4} b^{6} c^{2} + 3 a^{4} b^{4} c^{4} \\
		+ a^{6} b^{4} + a^{4} b^{6} + 8 a^{4} b^{4} c^{2} + a^{4} b^{2} c^{4} + a^{2} b^{4} c^{4} + 3 a^{4} b^{4} + 3 a^{4} b^{2} c^{2}\\ 
		\left.\left.+ 3 a^{2} b^{4} c^{2} + a^{2} b^{2} c^{4} + 2 a^{4} b^{2} + 2 a^{2} b^{4} + 4 a^{2} b^{2} c^{2} + 3 a^{2} b^{2} + a^{2} c^{2} + b^{2} c^{2} + a^{2} + b^{2} + c^{2} + 1
		\right]\right)
	\end{split}\end{align*}
	
	Then make the replacements using $a^2 = r$, $b^2 = s$, and $c^2 = t$. The resulting polynomial is
	$(r^{4} s^{4} t^{2} + r^{4} s^{4} t + r^{4} s^{3} t^{2} + r^{3} s^{4} t^{2} + r^{4} s^{3} t + r^{3} s^{4} t + 3  r^{3} s^{3} t^{2} + 4  r^{3} s^{3} t + 2  r^{3} s^{2} t^{2} + 2  r^{2} s^{3} t^{2} + r^{3} s^{3} + 3  r^{3} s^{2} t + 3  r^{2} s^{3} t + 3  r^{2} s^{2} t^{2} + r^{3} s^{2} + r^{2} s^{3} + 8  r^{2} s^{2} t + r^{2} s t^{2} + r s^{2} t^{2} + 3  r^{2} s^{2} + 3  r^{2} s t + 3  r s^{2} t + r s t^{2} + 2  r^{2} s + 2  r s^{2} + 4  r s t + 3  r s + r t + s t + r + s + t + 1) q^{2} t$, which matches exactly with the 66 terms from naively enumerating the plane partitions with the same weights. 
\end{example}

\subsection{Specializations}

In an attempt to find a simple application of these techniques to plane partition enumeration, we have noticed the following curious phenomenon.  Consider the specialization $a = b = c = \omega$, a primitive $n$th root of unity for various $n$, and let $L$ be a loop which appears in the $\SL_2$ double dimer model, after performing the squish map.

\begin{theorem}\label{firstroot}
	If $n = 1$ or $n = 2$, then the contribution of $L$ is the number of double dimer configurations that contribute to that loop. 
\end{theorem}

\begin{example}
	Consider the shape from Figure \ref{snake} (one of the six snake tiles). The path around this tile is $\gamma^{-1}(\beta^{-1}\alpha^{-1}\gamma\alpha^{-1})^2\gamma(\beta\alpha\gamma^{-1}\alpha)^2$, and if we calculate out that matrix product and then specialize to $a = b = c = \omega$, a first or second root of unity, we get the matrix $$\begin{bmatrix}
		1393 & 576 \\
		2208 & 913
	\end{bmatrix}.$$ The trace of this matrix is $2306$, corresponding with the 2306 possible double dimer configurations that squish to that particular snake tile. 
	
	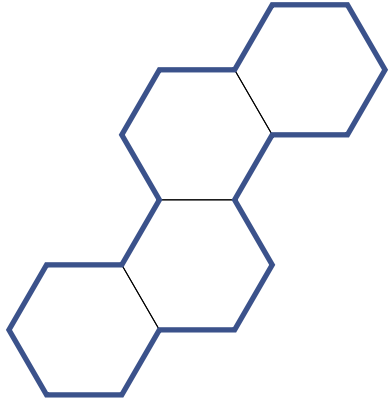
\begin{figure}[ht]
		$$\begin{tikzpicture}[scale = 0.5, x=1cm,y=1cm]
			\def \hex {+(0, 0) -- +(2, 0) -- +(3, 1*1.73) -- +(2, 2*1.73) -- +(0, 2*1.73) -- +(-1, 1*1.73) -- +(0, 0)};
			\draw[very thin] (-2, 0) \hex;
			\draw[very thin] (1, 1*1.73) \hex;
			\draw[very thin] (1, 3*1.73) \hex;
			\draw[very thin] (4, 4*1.73) \hex;
			\draw[color = blue, line width = 2pt] (0, 0) -- ++(-2, 0) -- ++(-1, 1.73) -- ++(1, 1.73) -- ++(2, 0) -- ++(1, 1.73) -- ++(-1, 1.73) -- ++(1, 1.73) -- ++(2, 0) -- ++(1, 1.73) -- ++(2, 0) -- ++(1, -1.73) -- ++(-1, -1.73) -- ++(-2, 0) -- ++(-1, -1.73) -- ++(1, -1.73) -- ++(-1, -1.73) -- ++(-2, 0) -- cycle;
		\end{tikzpicture}$$

		\caption{$\gamma^{-1}(\beta^{-1}\alpha^{-1}\gamma\alpha^{-1})^2\gamma(\beta\alpha\gamma^{-1}\alpha)^2$}
		\label{snake}
	\end{figure}
	\label{fourthroot}
\end{example}

\begin{proof}
	The first and second roots of unity case essentially acts as a reduction from the $\alpha$, $\beta$, and $\gamma$ version of the matrices to the original $L$ and $R$ vertex turn matrices. If we calculate the path around a shape as \emph{turns} instead of using edge connections $\alpha$, $\beta$, and $\gamma$, then we get the same resulting trace of the matrix product. In example \ref{fourthroot}, we would have the turn sequence (reading right-to-left since it corresponds to a path) $LLLRRLLRLLLLRRLLRL$, and if we use the originally-defined $L$ and $R$ matrices, we would get $\begin{bmatrix}337 & 1152\\
		576 & 1969\end{bmatrix}$, which has a trace of 2306, as in the example. 
\end{proof}

\begin{theorem} If $n=4$ then the contribution of $L$ is 2.
\end{theorem}

\begin{proof}
Note that this is the same contribution as if the connection were trivial (the matrix associated to every edge of $G$ being identity matrix). When $n = 4$, then each matrix $\alpha$, $\beta$, and $\gamma$ is equivalent to  $-I$, where $I$ is the $2 \times 2$ identity matrix. Then since any possible loop in the graph must have even length, the resulting product around a loop is $I$.
\end{proof}

\begin{theorem} \label{eighthroot}
	If $n = 8$, then the result of multiplying the connection around $L$ is $\pm \begin{bmatrix} 1 & 0 \\ 0 & 1 \end{bmatrix}$, where the sign tracks the parity of the number of hexagons contained in $L$. Note that this is still true regardless of our choice of basepoint. 
\end{theorem}

\begin{proof}
	The trace of a loop around a single hexagon when specialized to eighth roots of unity is $\begin{bmatrix}-1 & 0 \\ 0 & -1\end{bmatrix} = -1\cdot\begin{bmatrix}1 & 0 \\ 0 & 1\end{bmatrix} = -I$. Then assume that the path around a region containing $n$ hexagons is $(-1)^n\cdot I$. 
	
	Consider a region $R$ with $n$ hexagons, as in our assumption above. Attach one additional hexagon on the boundary of $R$ in a way that does not affect the simply-connectedness of $R$, and call this new region $R'$. Choose your basepoint to be the spot right after the added hexagon. Then the boundary of $R$ is the same until the new hexagon. Upon reaching the new hexagon, take the old path of $R$, then loop all the way around the new hexagon (still in the counterclockwise direction). Note that this will traverse the most recent one to five edges from $R$, but now in the opposite direction. Finally, the boundary of $R'$ also includes the boundary of the new hexagon. So our whole path is that of $R \cdot \{\text{path around a single hexagon}\}$, which is $(-1)^n\cdot I \cdot (-1)\cdot I = (-1)^{n+1}\cdot I$.
\end{proof}

For the next conjecture we need some terminology from Conway-Lagarias~\cite{conway-lagarias} about tiling regions in $H$ with tiles made of unions of hexagons.  A \emph{bone} is the union of three collinear adjacent hexagons in $H$; a \emph{stone} is the union of three hexagons in $H$ which all share a common vertex.  A \emph{signed tiling of $L$} is a collection of tiles, each with a weight of $+1$ or $-1$, covering $L$ in such a way that the total contribution at each hexagon inside $L$ is 1, and the total contribution of each hexagon outside $L$ is zero.  We also define one more tile, the \emph{snake}, which is a union of four hexagons in an ``S" shape (see Figure~\ref{snake}).

\begin{conjecture} 
\label{conj:stone bone snake}If $n=3$ or $n = 6$, then the contribution of $L$ is 0 unless there exists a signed tiling of $L$ by stones, bones and snakes -- in which case, the contribution is $(-1)^s \cdot I$, with $s$ being the number of stones used in the signed tiling.
\end{conjecture}

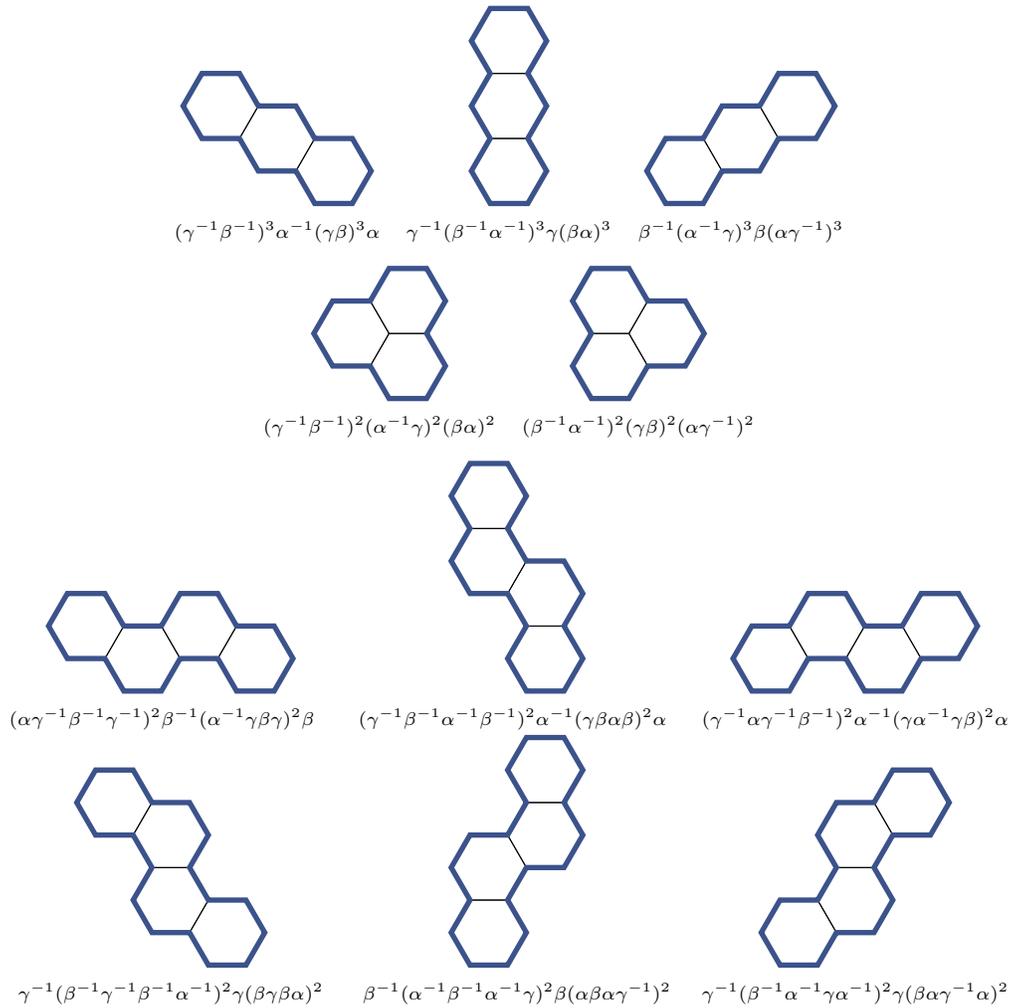
\begin{figure}[h]\small
	$$\begin{array}{c c c}
		\begin{tikzpicture}[scale = 0.25]
			\def \hex {+(0, 0) -- +(2, 0) -- +(3, 1*1.73) -- +(2, 2*1.73) -- +(0, 2*1.73) -- +(-1, 1*1.73) -- +(0, 0)};
			\draw[very thin] (0, 0)     \hex;
			\draw[very thin] (-3, 1*1.73) \hex;
			\draw[very thin] (-6, 2*1.73) \hex;
			\draw[color = blue, line width = 2pt] (0, 0) -- ++(2, 0) -- ++(1, 1.73) -- ++(-1, 1.73) -- ++(-2, 0) -- ++(-1, 1.73) -- ++(-2, 0) -- ++(-1, 1.73) -- ++(-2, 0) -- ++(-1, -1.73) -- ++(1, -1.73) -- ++(2, 0) -- ++(1, -1.73) -- ++(2, 0) -- ++(1, -1.73) -- cycle;
		\end{tikzpicture}
	&
		\begin{tikzpicture}[scale = 0.25]
			\def \hex {+(0, 0) -- +(2, 0) -- +(3, 1*1.73) -- +(2, 2*1.73) -- +(0, 2*1.73) -- +(-1, 1*1.73) -- +(0, 0)};
			\draw[very thin] (0, 0)     \hex;
			\draw[very thin] (0, 2*1.73) \hex;
			\draw[very thin] (0, 4*1.73) \hex;
			\draw[color = blue, line width = 2pt] (0, 0) -- ++(2, 0) -- ++(1, 1.73) -- ++(-1, 1.73) -- ++(1, 1.73) -- ++(-1, 1.73) -- ++(1, 1.73) -- ++(-1, 1.73) -- ++(-2, 0) -- ++(-1, -1.73) -- ++(1, -1.73) -- ++(-1, -1.73) -- ++(1, -1.73) -- ++(-1, -1.73) -- cycle;
		\end{tikzpicture}
	&	
		\begin{tikzpicture}[scale = 0.25]
			\def \hex {+(0, 0) -- +(2, 0) -- +(3, 1*1.73) -- +(2, 2*1.73) -- +(0, 2*1.73) -- +(-1, 1*1.73) -- +(0, 0)};
			\draw[very thin] (-2, 0)     \hex;
			\draw[very thin] (1, 1*1.73) \hex;
			\draw[very thin] (4, 2*1.73) \hex;
			\draw[color = blue, line width = 2pt] (0, 0) -- ++(-2, 0) -- ++(-1, 1.73) -- ++(1, 1.73) -- ++(2, 0) -- ++(1, 1.73) -- ++(2, 0) -- ++(1, 1.73) -- ++(2, 0) -- ++(1, -1.73) -- ++(-1, -1.73) -- ++(-2, 0) -- ++(-1, -1.73) -- ++(-2, 0) -- ++(-1, -1.73) -- cycle;
		\end{tikzpicture}
	\\
		\scriptstyle(\gamma^{-1}\beta^{-1})^3\alpha^{-1}(\gamma\beta)^3 \alpha 
	&
		\scriptstyle\gamma^{-1}(\beta^{-1}\alpha^{-1})^3\gamma(\beta\alpha)^3 
	& 
		\scriptstyle\beta^{-1}(\alpha^{-1}\gamma)^3 \beta (\alpha\gamma^{-1})^3
	\end{array}$$
	$$\begin{array}{c c}
		\begin{tikzpicture}[scale = 0.25]
			\def \hex {+(0, 0) -- +(2, 0) -- +(3, 1*1.73) -- +(2, 2*1.73) -- +(0, 2*1.73) -- +(-1, 1*1.73) -- +(0, 0)};
			\draw[very thin] (0, 0)     \hex;
			\draw[very thin] (3, -1*1.73) \hex;
			\draw[very thin] (3, 1*1.73) \hex;
			\draw[color = blue, line width = 2pt] (0, 0) -- ++(2, 0) -- ++(1, -1.73) -- ++(2, 0) -- ++(1, 1.73) -- ++(-1, 1.73) -- ++(1, 1.73) -- ++(-1, 1.73) -- ++(-2, 0) -- ++(-1, -1.73) -- ++(-2, 0) -- ++(-1, -1.73) -- ++(1, -1.73) -- cycle;
		\end{tikzpicture}
	&
		\begin{tikzpicture}[scale = 0.25]
			\def \hex {+(0, 0) -- +(2, 0) -- +(3, 1*1.73) -- +(2, 2*1.73) -- +(0, 2*1.73) -- +(-1, 1*1.73) -- +(0, 0)};
			\draw[very thin] (-2, 0)     \hex;
			\draw[very thin] (-5, -1*1.73) \hex;
			\draw[very thin] (-5, 1*1.73) \hex;
			\draw[color = blue, line width = 2pt] (0, 0) -- ++(-2, 0) -- ++(-1, -1.73) -- ++(-2, 0) -- ++(-1, 1.73) -- ++(1, 1.73) -- ++(-1, 1.73) -- ++(1, 1.73) -- ++(2, 0) -- ++(1, -1.73) -- ++(2, 0) -- ++(1, -1.73) -- ++(-1, -1.73) -- cycle;
		\end{tikzpicture}
	\\
		\scriptstyle(\gamma^{-1}\beta^{-1})^2(\alpha^{-1}\gamma)^2(\beta\alpha)^2 
	&
		\scriptstyle(\beta^{-1}\alpha^{-1})^2(\gamma\beta)^2(\alpha\gamma^{-1})^2
	\end{array}$$
	$$\begin{array}{c c c}
		\begin{tikzpicture}[scale = 0.25]
			\def \hex {+(0, 0) -- +(2, 0) -- +(3, 1*1.73) -- +(2, 2*1.73) -- +(0, 2*1.73) -- +(-1, 1*1.73) -- +(0, 0)};
			\draw[very thin] (0, -2*1.73) \hex;
			\draw[very thin] (3, -3*1.73) \hex;
			\draw[very thin] (6, -2*1.73) \hex;
			\draw[very thin] (9, -3*1.73) \hex;
			\draw[color = blue, line width = 2pt] (0, 0) -- ++(2, 0) -- ++(1, -1.73) -- ++(2, 0) -- ++(1, 1.73) -- ++(2, 0) -- ++(1, -1.73) -- ++(2, 0) -- ++(1, -1.73)  -- ++(-1, -1.73) -- ++(-2, 0) -- ++(-1, 1.73) -- ++(-2, 0) -- ++(-1, -1.73) -- ++(-2, 0) -- ++(-1, 1.73) -- ++(-2, 0) -- ++(-1, 1.73) -- cycle;
		\end{tikzpicture}
	&
		\begin{tikzpicture}[scale = 0.25]
			\def \hex {+(0, 0) -- +(2, 0) -- +(3, 1*1.73) -- +(2, 2*1.73) -- +(0, 2*1.73) -- +(-1, 1*1.73) -- +(0, 0)};
			\draw[very thin] (-2, 0) \hex;
			\draw[very thin] (-2, 2*1.73) \hex;
			\draw[very thin] (-5, 3*1.73) \hex;
			\draw[very thin] (-5, 5*1.73) \hex;
			\draw[color = blue, line width = 2pt] (0, 0) -- ++(-2, 0) -- ++(-1, 1.73) -- ++(1, 1.73) -- ++(-1, 1.73) -- ++(-2, 0) -- ++(-1, 1.73) -- ++(1, 1.73) -- ++ (-1, 1.73) -- ++(1, 1.73) -- ++(2, 0) -- ++(1, -1.73) -- ++(-1, -1.73) -- ++(1, -1.73) -- ++(2, 0) -- ++(1, -1.73) -- ++(-1, -1.73) -- ++(1, -1.73) -- cycle;
			\end{tikzpicture}
	&
		\begin{tikzpicture}[scale = 0.25]
			\def \hex {+(0, 0) -- +(2, 0) -- +(3, 1*1.73) -- +(2, 2*1.73) -- +(0, 2*1.73) -- +(-1, 1*1.73) -- +(0, 0)};
			\draw[very thin] (0, 0) \hex;
			\draw[very thin] (3, 1*1.73) \hex;
			\draw[very thin] (6, 0*1.73) \hex;
			\draw[very thin] (9, 1*1.73) \hex;
			\draw[color = blue, line width = 2pt] (0, 0) -- ++(2, 0) -- ++(1, 1.73) -- ++(2, 0) -- ++(1, -1.73) -- ++(2, 0) -- ++(1, 1.73) -- ++(2, 0) -- ++(1, 1.73)  -- ++(-1, 1.73) -- ++(-2, 0) -- ++(-1, -1.73) -- ++(-2, 0) -- ++(-1, 1.73) -- ++(-2, 0) -- ++(-1, -1.73) -- ++(-2, 0) -- ++(-1, -1.73) -- cycle;
		\end{tikzpicture}
	\\
		\scriptstyle(\alpha\gamma^{-1}\beta^{-1}\gamma^{-1})^2\beta^{-1}(\alpha^{-1}\gamma\beta\gamma)^2\beta\ \  
	&
		 \scriptstyle(\gamma^{-1}\beta^{-1}\alpha^{-1}\beta^{-1})^2\alpha^{-1}(\gamma\beta\alpha\beta)^2\alpha\ 
 	&
		\scriptstyle(\gamma^{-1}\alpha\gamma^{-1}\beta^{-1})^2\alpha^{-1}(\gamma\alpha^{-1}\gamma\beta)^2\alpha 
	\\
		\begin{tikzpicture}[scale = 0.25]
			\def \hex {+(0, 0) -- +(2, 0) -- +(3, 1*1.73) -- +(2, 2*1.73) -- +(0, 2*1.73) -- +(-1, 1*1.73) -- +(0, 0)};
			\draw[very thin] (0, 0) \hex;
			\draw[very thin] (-3, 1*1.73) \hex;
			\draw[very thin] (-3, 3*1.73) \hex;
			\draw[very thin] (-6, 4*1.73) \hex;
			\draw[color = blue, line width = 2pt] (0, 0) -- ++(2, 0) -- ++(1, 1.73) -- ++(-1, 1.73) -- ++(-2, 0) -- ++(-1, 1.73) -- ++(1, 1.73) -- ++(-1, 1.73) -- ++(-2, 0) -- ++(-1, 1.73) -- ++(-2, 0) -- ++(-1, -1.73) -- ++(1, -1.73) -- ++(2, 0) -- ++(1, -1.73) -- ++(-1, -1.73) -- ++(1, -1.73) -- ++(2, 0) -- cycle;
			\end{tikzpicture}
	&
		\begin{tikzpicture}[scale = 0.25]
		\def \hex {+(0, 0) -- +(2, 0) -- +(3, 1*1.73) -- +(2, 2*1.73) -- +(0, 2*1.73) -- +(-1, 1*1.73) -- +(0, 0)};
		\draw[very thin] (0, 0) \hex;
		\draw[very thin] (0, 2*1.73) \hex;
		\draw[very thin] (3, 3*1.73) \hex;
		\draw[very thin] (3, 5*1.73) \hex;
		\draw[color = blue, line width = 2pt] (0, 0) -- ++(2, 0) -- ++(1, 1.73) -- ++(-1, 1.73) -- ++(1, 1.73) -- ++(2, 0) -- ++(1, 1.73) -- ++(-1, 1.73) -- ++ (1, 1.73) -- ++(-1, 1.73) -- ++(-2, 0) -- ++(-1, -1.73) -- ++(1, -1.73) -- ++(-1, -1.73) -- ++(-2, 0) -- ++(-1, -1.73) -- ++(1, -1.73) -- ++(-1, -1.73) -- cycle;
		\end{tikzpicture}
	&
		\begin{tikzpicture}[scale = 0.25]
			\def \hex {+(0, 0) -- +(2, 0) -- +(3, 1*1.73) -- +(2, 2*1.73) -- +(0, 2*1.73) -- +(-1, 1*1.73) -- +(0, 0)};
			\draw[very thin] (-2, 0) \hex;
			\draw[very thin] (1, 1*1.73) \hex;
			\draw[very thin] (1, 3*1.73) \hex;
			\draw[very thin] (4, 4*1.73) \hex;
			\draw[color = blue, line width = 2pt] (0, 0) -- ++(-2, 0) -- ++(-1, 1.73) -- ++(1, 1.73) -- ++(2, 0) -- ++(1, 1.73) -- ++(-1, 1.73) -- ++(1, 1.73) -- ++(2, 0) -- ++(1, 1.73) -- ++(2, 0) -- ++(1, -1.73) -- ++(-1, -1.73) -- ++(-2, 0) -- ++(-1, -1.73) -- ++(1, -1.73) -- ++(-1, -1.73) -- ++(-2, 0) -- cycle;
			\end{tikzpicture}
	\\
		\scriptstyle\gamma^{-1}(\beta^{-1}\gamma^{-1}\beta^{-1}\alpha^{-1})^2\gamma(\beta\gamma\beta\alpha)^2 
	&
		\scriptstyle\beta^{-1}(\alpha^{-1}\beta^{-1}\alpha^{-1}\gamma)^2\beta(\alpha\beta\alpha\gamma^{-1})^2 
	& 
		\scriptstyle\gamma^{-1}(\beta^{-1}\alpha^{-1}\gamma\alpha^{-1})^2\gamma(\beta\alpha\gamma^{-1}\alpha)^2
		
	\end{array}$$\normalsize
	\caption{All of the tiles we use for $\omega$ a third or sixth root of unity. The top row is all three orientation of bones, the second row both orientation of stones, and the bottom two rows all six orientations of the snake tile.}
	\label{alltiles}
\end{figure}

When $n = 3$ or $n = 6$, then the contribution of a loop around a bone or snake is $I$. So if we inductively add tiles to a region (as we did in the proof of Theorem \ref{eighthroot} for individual hexagons), we get $\pm I$, where the parity depends only on the number of stone tiles used in the tiling.  So to prove the conjecture, one needs only show that the monodromy of all non-signed-tilable $L$ is 0.  

We expect that after this specialization our $\SL_2$ connection is strongly related to the character of one of the Conway-Lagarias tiling groups, so the proof will involve computing this character, as well as a map from said tiling group to a subgroup generated by the matrices $\alpha$, $\beta$, and $\gamma$.  Note that a contribution of 0 means that a double dimer configuration will not contribute at all to the partition function, so this should give an interesting generating function for ``good pairs" of plane partitions.

\section{Concluding remarks}

\begin{figure}
	$$\begin{tikzpicture}
		\clip[draw] (0.25,0) circle (3.5);
		\node at (0,0) {\includegraphics[width = .75\textwidth]{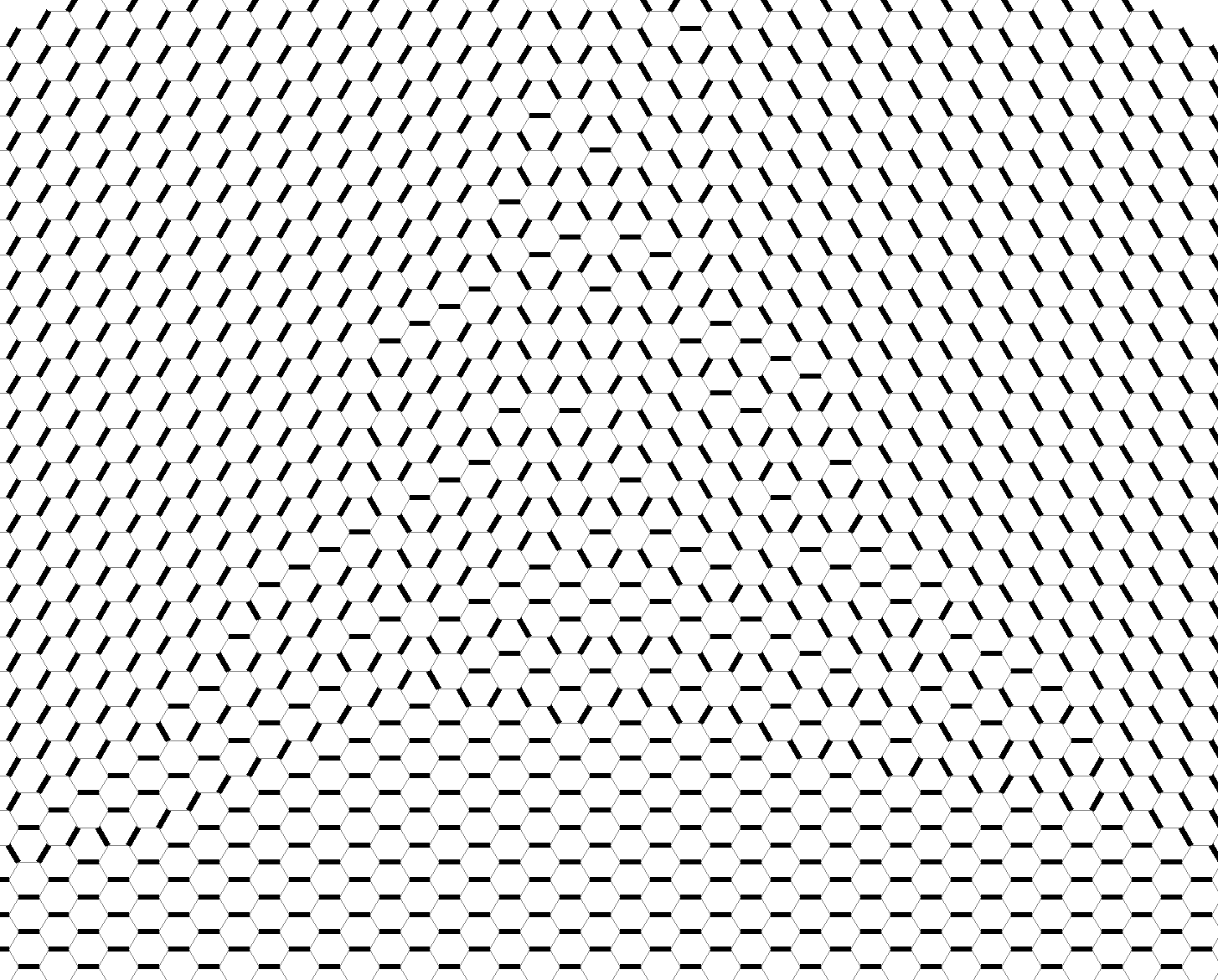}};
	\end{tikzpicture}
	\ \ 
	\begin{tikzpicture}
		\clip[draw] (0.25,.4) circle (3.5);
		\node at (0,0) {\includegraphics[width = .65\textwidth]{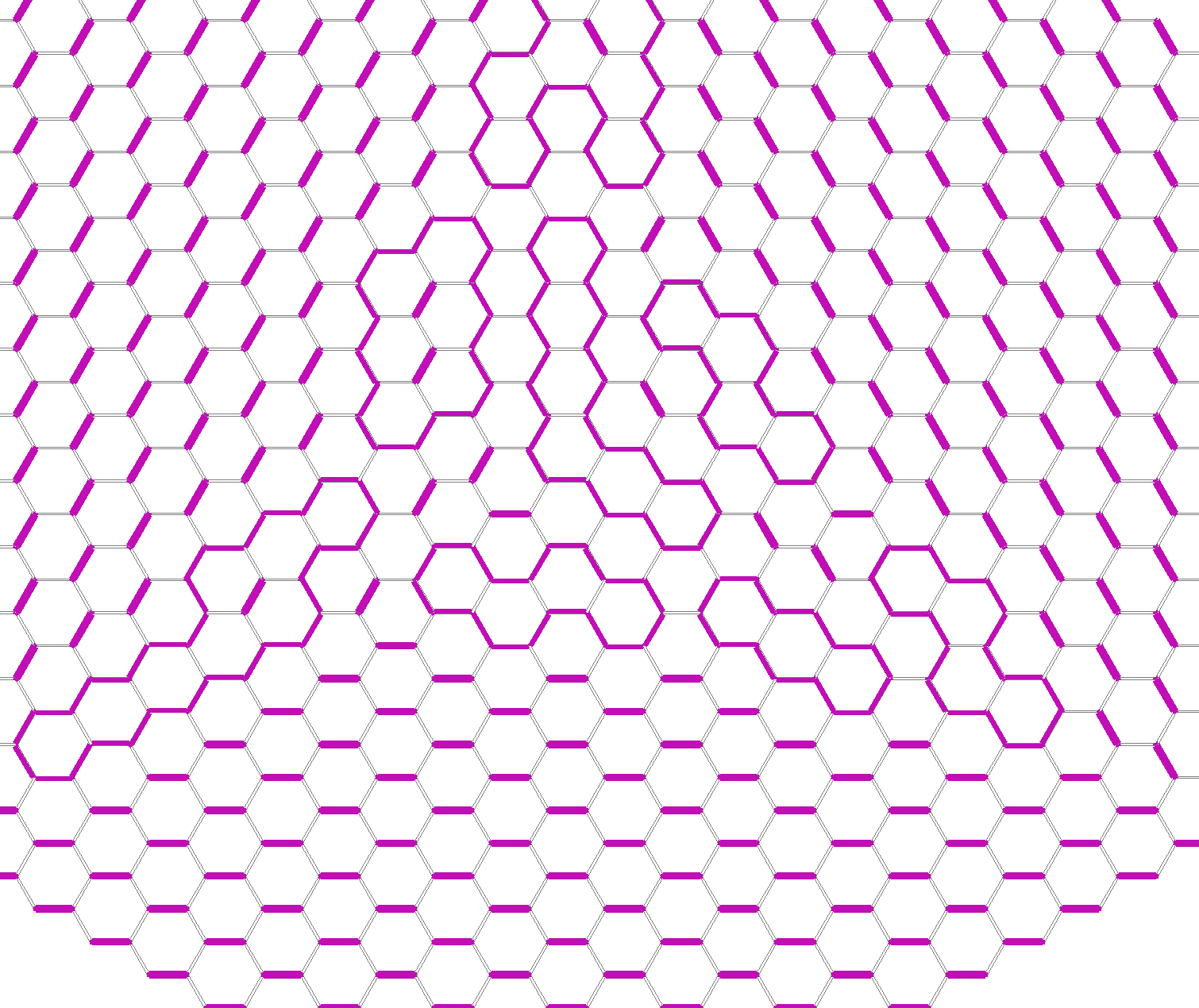}};
	\end{tikzpicture}$$
	\caption{The single dimer model before and after the squish map showing various closed loops, including two snakes, a bone, and some other loops with positive contributions under Conjecture~\ref{conj:stone bone snake}}
	\label{loops}
\end{figure}

	Here we explain the process of finding $\alpha$, $\beta$, and $\gamma$. Perhaps our discovery process is of interest to the reader that has made it this far.  
	
	Let $M$ (for move) be the matrix $M = L^{-1}R$. Then $M^{-1} = R^{-1}L$. Our first attempt to create matrices $\alpha$, $\beta$, and $\gamma$ involved simply conjugating $A,B,C$ with different powers of $M$, noting that $M^3 =-I$. Based on empirical studies of the shortest possible loop, six left turns around a single hexagon, this looked close to the correct definition.  We defined $\alpha_1 := A$, $\beta_1 := M^{-1}BM$, and $\gamma_1 := M^{-2}BM^2$. These matrices worked for our main goal -- terms in the trace of the resulting matrix count how many perfect matchings squish to a particular loop, with the correct weights. However, that polynomial had alternating signs depending on the total degree of each term.
	
	Our next step was to correct for the signs by replacing $a$ with $ai$, $b$ with $bi$, and $c$ with $ci$ in the original $A$, $B$ and $C$ matrices. Let $A'$, $B'$, and $C'$ be the corresponding $A$, $B$, and $C$ matrices after making the above replacements. We wanted all of the signs to be positive because the matrix product of a path around a loop should give rise to a probability measure on the graph, as in \cite{kenyon:conformal}. After including the $i$s we had a trace that has all positive values. Then, $\alpha_2 := A'$, $\beta_2 := M^{-1}B'M$, and $\gamma_2 := M^{-2}C'M^2$. 
	
	We had one final problem, however: we needed to accomplish the specialization $a \mapsto ai$, \emph{etc}, using linear algebra.  We did this using the matrix $J$ as defined above. We can now rewrite $A' = iJA$, $B' = -iJB$, and $C' = iJC$.  Finally, to correct for the signs we mentioned, we conjugate each $\alpha_2$, $\beta_2$, and $\gamma_2$ with $J$. So (finally) we have that $\alpha_3 := JA'J$, $\beta_3 := JM^{-1}B'MJ$, and $\gamma_3 := JM^{-2}C'M^2J$. After rewriting and simplifying, we arrive at the final definition of $\alpha$, $\beta$, and $\gamma$ as stated in Definition \ref{matrices}. 
	
\subsection*{Acknowledgements}

We would like to thank Richard Kenyon and Sunil Chhita for helpful conversations. Leigh Foster received support from NSF grant DMS-2039316.


\begin{thebibliography}{1}

\bibitem{conway-lagarias}
J.~H. Conway and J.~C. Lagarias.
\newblock Tiling with polyominoes and combinatorial group theory.
\newblock {\em J. Combin. Theory Ser. A}, 53(2):183--208, 1990.

\bibitem{gessel-viennot}
Ira Gessel and G\'{e}rard Viennot.
\newblock Binomial determinants, paths, and hook length formulae.
\newblock {\em Adv. in Math.}, 58(3):300--321, 1985.

\bibitem{kasteleyn:statistics}
PW~Kasteleyn.
\newblock The statistics of dimers on a lattice.
\newblock {\em Physica}, 27:1209--1225, 1961.

\bibitem{kenyon:conformal}
Richard Kenyon.
\newblock Conformal invariance of loops in the double-dimer model.
\newblock {\em Comm. Math. Phys.}, 326(2):477--497, 2014.

\bibitem{lindstrom}
Bernt Lindstr\"{o}m.
\newblock On the vector representations of induced matroids.
\newblock {\em Bull. London Math. Soc.}, 5:85--90, 1973.

\bibitem{macmahon}
Percy~A. MacMahon.
\newblock {\em Combinatory analysis. {V}ol. {I}, {II} (bound in one volume)}.
\newblock Dover Phoenix Editions. Dover Publications Inc., Mineola, NY, 2004.

\bibitem{temperley-fisher}
H.~N.~V. Temperley and Michael~E. Fisher.
\newblock Dimer problem in statistical mechanics---an exact result.
\newblock {\em Philos. Mag. (8)}, 6:1061--1063, 1961.

\bibitem{young:squish}
Benjamin Young.
\newblock Squishing dimers on the hexagon lattice.
\newblock {\em Electron. J. Combin.}, 16(1):Research Paper 86, 20, 2009.

\bibitem{young:thesis}
Benjamin Young.
\newblock Generating functions for colored 3{D} {Y}oung diagrams and the
  {D}onaldson-{T}homas invariants of orbifolds.
\newblock {\em Duke Math. J.}, 152(1):115--153, 2010.
\newblock With an appendix by Jim Bryan.

\end{thebibliography}

\end{document}